\documentclass[12pt,a4paper]{article}
\usepackage{amsmath}
\usepackage{amsfonts}
\usepackage{amssymb}
\usepackage{graphicx}
\usepackage{tikz}
\usepackage{amsthm} 
\usepackage{gensymb} 
\usepackage{subcaption}
\usepackage{hyperref}
\usepackage{cleveref}
\usepackage{xcolor} 
\hypersetup{colorlinks, linkcolor={blue}, citecolor={green!80!black}, urlcolor={blue}}
\usepackage{authblk}
\usepackage{comment}
\usepackage{mathtools}
\usepackage{enumitem} 
\usepackage{float} 

\newtheorem{theorem}{Theorem}[section]
\newtheorem{lemma}[theorem]{Lemma}
\newtheorem{proposition}[theorem]{Proposition}
\newtheorem{corollary}[theorem]{Corollary}
\newtheorem{definition}{Definition}[section]

\newtheorem*{theorem*}{Theorem}
\newtheorem*{corollary*}{Corollary}
\newtheorem{fact}{Fact}[section]

\allowdisplaybreaks 
\newcommand{\instar}{\textnormal{\hbox{is}}}
\newcommand{\aface}{\textnormal{\hbox{af}}}
\newcommand{\cface}{\textnormal{\hbox{cf}}}
\newcommand{\abs}[1]{\lvert#1\rvert}
\crefformat{section}{\S#2#1#3} 

\makeatletter
\renewcommand{\make@df@tag@@@}[2][]{%
	\gdef\df@tag{%
		\tagform@{#2\rlap{\hphantom)#1}}%
		\toks@\@xp{\p@equation{#2}}%
		\edef\@currentlabel{\the\toks@}%
	}%
}
\makeatother

\title{Tutte Invariants for Alternating Dimaps}
\author[1,3]{Kai Siong Yow\thanks{Email: \texttt{ksyow@upm.edu.my} (majority of the work was done at Monash University Australia as part of the author's PhD research)}}
\author[1]{Graham Farr\thanks{Email: \texttt{graham.farr@monash.edu}}}
\author[2]{Kerri Morgan\thanks{Email: \texttt{kerri.morgan@deakin.edu.au}}}
\affil[1]{Faculty of Information Technology, Monash University, Clayton, Victoria 3800, Australia}
\affil[2]{Deakin University, Geelong, Australia, School of Information Technology, Faculty of Science Engineering \& Built Environment}
\affil[3]{Department of Mathematics, Faculty of Science, Universiti Putra Malaysia, 43400 UPM Serdang, Selangor, Malaysia}

\date{\today}

\begin{document}
	
\maketitle

\begin{abstract}
	An \emph{alternating dimap} is an orientably embedded Eulerian directed graph where the edges incident with each vertex are directed inwards and outwards alternately. Three reduction operations for alternating dimaps were investigated by Farr. A \emph{minor} of an alternating dimap can be obtained by reducing some of its edges using the reduction operations. Unlike classical minor operations, these reduction operations do not commute in general. A \emph{Tutte invariant} for alternating dimaps is a function $ P $ defined on every alternating dimap and taking values in a field such that $ P $ is invariant under isomorphism and obeys a linear recurrence relation involving reduction operations. It is well known that if a graph $ G $ is planar, then the Tutte polynomial $ T $ satisfies $ T(G;x,y)=T(G^{*};y,x) $. We note an analogous relation for the extended Tutte invariants for alternating dimaps introduced by Farr. We then characterise the Tutte invariant for alternating dimaps of genus zero under several conditions. As a result of the non-commutativity of the reduction operations, the recursions based on them cannot always be satisfied. We investigate the properties of alternating dimaps of genus zero that are required in order to obtain a well defined Tutte invariant. Some excluded minor characterisations for these alternating dimaps are also given.
\end{abstract}


\section{Introduction}\label{sec:introduction}

Alternating dimaps were introduced by Tutte in 1948~\cite{Tutte1948} as a tool for studying the dissection of equilateral triangles into equilateral triangles. This followed his collaboration with Brooks, Smith and Stone on dissecting rectangles into squares~\cite{BSSTutte1940}. The concept of duality in graphs was extended to a higher order transform, namely \emph{triality}, in alternating dimaps~\cite{Tutte1948,Tutte1973}.

Recently, Farr~\cite{Farr2013pp,Farr2018} introduced three minor operations for alternating dimaps, known as $ 1 $-reduction, $ \omega $-reduction and $ \omega^{2} $-reduction. A \emph{minor} of an alternating dimap can be obtained by reducing some of its edges using these operations. Unlike classical minor operations, these operations do not always commute.

Alternating dimaps are at least as diverse and rich in structure as orientably embedded undirected graphs, since any orientably embedded undirected graph can be converted into an alternating dimap by replacing each edge by a clockwise 2-cycle. It is also worth noting that every orientably embedded Eulerian undirected graph with $ k $ components can be converted into an alternating dimap in $ 2^{k} $ ways: for each component, choose a reference edge arbitrarily, choose one of the two possible directions of it, and let the alternating property determine the direction of all other edges. So, given the breadth of the class of alternating dimaps, and the natural reduction operations for it, it is natural to ask whether a theory of Tutte polynomials may be developed for it.

The first steps in this direction were taken by Farr in~\cite{Farr2013pp,Farr2018}. In this paper, we develop the theory further, establishing when various types of Tutte invariant exist, with both positive and negative results.

A \emph{Tutte invariant} for some class of alternating dimaps is a function $ P $ defined on every alternating dimap in the class and taking values in a field $ \mathbb{F} $ such that $ P $ is invariant under isomorphism and obeys a linear recurrence relation involving reduction operations. Farr~\cite{Farr2018} defined some Tutte invariants including extended Tutte invariants, the c-Tutte invariant and the a-Tutte invariant. We characterise these invariants in this paper.

When $ G $ is a planar graph and $ G^{*} $ is the dual graph of $ G $, the Tutte polynomial $ T(G;x,y) $ of $ G $ satisfies $ T(G;x,y)=T(G^{*};y,x) $. We prove an analogous result for extended Tutte invariants for alternating dimaps.
	
Since the reduction operations do not always commute, an invariant defined recursively using reductions may not always exist. We investigate the conditions required to obtain a well defined Tutte invariant for alternating dimaps of genus zero. First, we determine the form that an extended Tutte invariant must have, if it is to be well defined for all alternating dimaps of genus zero. It turns out that such invariants are of very restricted form. Then, we determine the structure of alternating dimaps of genus zero for which extended Tutte invariants are as unrestricted in form as possible. We also establish some excluded minor characterisations for those alternating dimaps.

For any embedded graph $ G $, its associated alternating dimap $ \hbox{alt}_c(G) $ (respectively, $ \hbox{alt}_a(G) $) is obtained by replacing each edge of $ G $ by a pair of directed edges forming a clockwise face (respectively, anticlockwise face) of size two~\cite{Farr2018}. The c-Tutte invariant $ T_c(D;x,y) $ of an alternating dimap $ D $ was introduced in~\cite{Farr2018} and shown to be well defined for any alternating dimap of the form $ \hbox{alt}_c(G) $ where $ G $ is a plane graph, when it equals the Tutte polynomial of $ G $. But it can be defined for some other alternating dimaps too. We determine the class of alternating dimaps for which the c-Tutte invariant is well defined. It properly contains alternating dimaps of the form $ \hbox{alt}_c(G) $, where $ G $ is a plane graph. This shows that the c-Tutte invariants properly extend the Tutte polynomial of a plane graph. Analogous results are established for a-Tutte invariants and $ \hbox{alt}_a(G) $.

As is common when working with recursively defined invariants, our proofs are mostly by induction. It is worth noting that the theory of alternating dimaps can be formulated in terms of bicubic maps \cite{Tutte1948}. Each framework has its merits: in alternating dimaps, ordinary graph-theoretic contraction appears in one of the operations, while in bicubic maps, triality has a simple interpretation in terms of cycling through the three available colours. We use the alternating dimap framework for the proofs in this paper because the graph concepts used are closer to those found elsewhere in graph theory.

Earlier versions of this work are in~\cite{YowFM2018pp} (v1 with full proofs, v2 with some details omitted) and the first author's PhD thesis \cite[Chapters 3,5,6]{YowPHD2019}.

\section{Definitions and Notation}\label{sec:definitions}

Our notation is mostly standard. Alternating dimap notation and terminology is as in \cite{Farr2018}. We let $ D $ be an alternating dimap throughout this section, unless otherwise stated.

An embedded graph is \emph{2-cell embedded} if each of its faces is homeomorphic to an open unit disc.

A \emph{dimap} is an embedded directed graph that is 2-cell embedded in a disjoint union of orientable surfaces (or 2-manifolds).

An \emph{alternating dimap} is a dimap without isolated vertices where, for each vertex, the sequence of edges incident with it is directed inwards and outwards alternately in a cyclic order around the vertex. All vertices have even degree. An alternating dimap may have loops and multiple edges. If an alternating dimap contains no vertices, edges, or faces, it is the \emph{empty alternating dimap}. Let $ u,v \in V(D) $. We use $ uv $ to represent an edge directed from $ u $ to $ v $, hence $ uv \ne vu $ in this context. We call $ v $ the \emph{head} and $ u $ the \emph{tail} of the edge. The sets of vertices, edges, and faces, and the number of connected components of $ D $, are denoted by $ V(D) $, $ E(D) $, $ F(D) $, and $ k(D) $, respectively. The number of clockwise faces (c-faces) and the number of anticlockwise faces (a-faces) of $ D $ are denoted by $ \cface(D) $ and $ \aface(D) $, respectively. An \emph{in-star} is the set of edges directed into a vertex. The in-star that is directed into a vertex $ v $ is denoted by $ I(v) $. The number of in-stars of $ D $ is denoted by $ \instar(D) $. The next edge after $ e \in E(D) $ going around a clockwise face (respectively, an anticlockwise face) in the direction indicated by $ e $ is the \emph{right successor} (respectively, \emph{left successor}) of $ e $.

Two alternating dimaps $ D_{1} $ and  $ D_{2} $ are \emph{isomorphic}, written as $ D_{1} \cong D_{2} $, if their underlying directed graphs are isomorphic and there exists a bijection $ \phi:E(D_{1}) \rightarrow E(D_{2}) $ such that $ e \in E(D_{1}) $ has $ f $ and $ g $ as its right and left successors respectively, if and only if $ \phi(e) \in E(D_{2}) $ has $ \phi(f) $ and $ \phi(g) $ as its right and left successors, respectively.

We say $ D' $ is an \emph{alternating subdimap} of $ D $, written as $ D' \le D $, if $ D' $ is an alternating dimap where $ V(D') \subseteq V(D) $ and $ E(D') \subseteq E(D) $.

Suppose $ e \in E(D) $. We write $ D/e $ and $ D\setminus e $ for the alternating dimap, or dimap, obtained from $ D $ by contracting $ e $ and by deleting $ e $, respectively\footnote{A dimap (instead of an alternating dimap) $ D\setminus e $ is obtained by deleting a non-loop edge $ e $ from an alternating dimap $ D $. In this scenario, both of the endvertices of $ e $ have odd degree in $ D\setminus e $, hence $ D\setminus e $ is no longer an alternating dimap. See \cref{sec:type_of_loops} for edge types in alternating dimaps.}.

In any embedded graph, the \emph{boundary} $ \partial g $ of a face $ g $ is the closed trail that bounds $ g $. A face is \emph{incident} with every vertex and every edge that belongs to its boundary. Two faces are \emph{adjacent} if their boundaries share at least one common edge.

By the alternating property, every face in $ D $ is bounded either by an embedding of (the graph of) a clockwise closed trail or an anticlockwise closed trail. Let $ C $ be the closed trail bounding a face in $ D $. The alternating subdimap \emph{induced} by $ C $, written as $ D[C] $, is the alternating subdimap of $ D $ with the vertex set $ V(C) $ and the edge set $ E(C) $ (see Figure~\ref{fig:Induced_alternating_subdimap}).
\begin{figure}[t]
	\centering
	\includegraphics[width=0.75\textwidth]{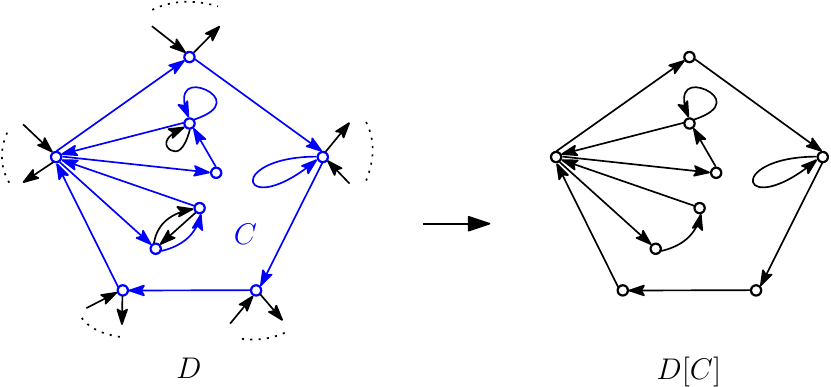}
	\caption{An alternating subdimap induced by a closed trail $ C $ (shown in blue) bounding a face in $ D $}
	\label{fig:Induced_alternating_subdimap}
\end{figure}

A \emph{block} of an alternating dimap is a maximal connected alternating subdimap that contains no cutvertex.

The \emph{closure} of a subset $ S $ of points in a topological space is the union of $ S $ and its boundary.

Suppose the plane alternating dimap $ D $ contains blocks $ B_{1} $ and $ B_{2} $, and let $ g \in F(B_{2}) $. The block $ B_{1} $ is \emph{within}\label{def:within} the face $ g $ if the point set formed by the embeddings of $ V(B_{1}) $ and $ E(B_{1}) $ is a subset of the closure of $ g $. In Figure~\ref{fig:Within_blocks}(a), the block $ B_{1} $ (highlighted in green) is within the face $ g \in F(B_{2}) $, whereas the block $ B_{3} $ (highlighted in blue) is within the face $ h \in F(B_{2}) $ (and not within $ g $). The faces $ g,h \in F(B_{2}) $ are shown in Figure~\ref{fig:Within_blocks}(b).
\begin{figure}[t]
	\centering
	\begin{subfigure}[t]{0.4\textwidth}
		\includegraphics[width=\textwidth]{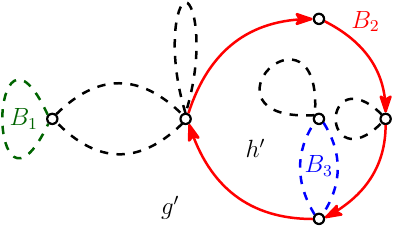}
		\caption{$ D $}
	\end{subfigure}\qquad \qquad
	~
	\begin{subfigure}[t]{0.23\textwidth}
		\includegraphics[width=\textwidth]{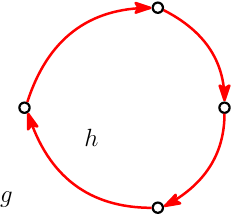}
		\caption{$ B_{2} $}
	\end{subfigure}
	\caption{(a) Blocks within faces of the block $ B_{2} $ in an alternating dimap $ D $ where $ B_{2} $ is a directed cycle and other blocks are shown schematically using dashed lines, (b) The block $ B_{2} $ in $ D $}
	\label{fig:Within_blocks}
\end{figure}

Let $ D_{1} $ and $ D_{2} $ be two plane alternating dimaps, let $ a_{1} $ and $ a_{2} $ be anticlockwise faces of $ D_{1} $ and $ D_{2} $ respectively, and let $ v_{1} \in V(\partial a_{1}) $ and $ v_{2} \in V(\partial a_{2}) $ be vertices of the faces $ a_{1} $ and $ a_{2} $. The \emph{c-union}\label{def:c-union} $ D $ of $ D_{1} $ and $ D_{2} $ \emph{with respect to} $ a_{1}, v_{1}, a_{2} $ and $ v_{2} $, denoted by $ D_{1} \cup_{c} D_{2} $, is obtained by identifying $ v_{1} $ and $ v_{2} $ such that $ D_{1} $ is within the anticlockwise face $ a_{2} $ of $ D_{2} $, and $ D_{2} $ is within the anticlockwise face $ a_{1} $ of $ D_{1} $, in $ D $. When the context is clear, we may just refer to the \emph{c-union} $ D $ of $ D_{1} $ and $ D_{2} $. Note that the set of clockwise faces of $ D $ is the union of the sets of clockwise faces of $ D_{1} $ and $ D_{2} $ (hence the term c-union), and $ \abs{E(D)}=\abs{E(D_{1})}+\abs{E(D_{2})} $. An example of a c-union of two alternating dimaps is shown in Figure~\ref{fig:c-union}. The \emph{a-union} is defined by appropriate modifications.
\begin{figure}[t]
	\centering
	\begin{subfigure}[t]{0.9\textwidth}
		\includegraphics[width=\textwidth]{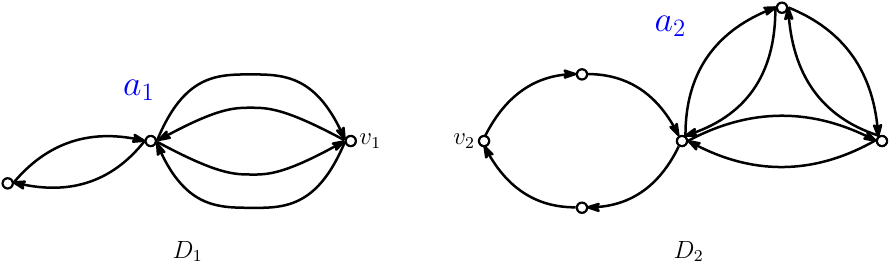}
		\caption{Two alternating dimaps $ D_{1} $ and $ D_{2} $}
	\end{subfigure}
	
	\begin{subfigure}[t]{0.75\textwidth}
		\includegraphics[width=\textwidth]{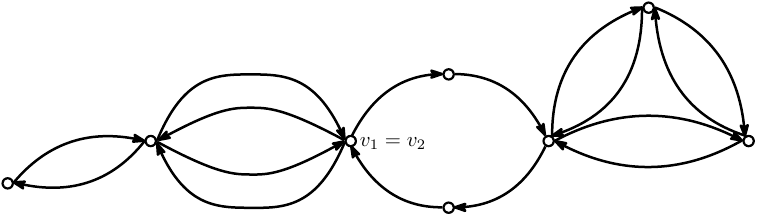}
		\caption{$ D = D_{1} \cup_{c} D_{2} $}
	\end{subfigure}
	\caption{A c-union of two alternating dimaps. For convenience, the anticlockwise faces $ a_{1} \in F(D_{1})$ and $ a_{2} \in F(D_{2}) $ are both shown as outer regions}
	\label{fig:c-union}
\end{figure}

A \emph{plane alternating dimap} $ \hbox{Pl}(D) $ of an alternating dimap $ D $ of genus zero is obtained from $ D $ by converting its embedding in the sphere into an embedding in the plane, by stereographic projection in the usual way~\cite{Lawson2003}.

The \emph{face-rooted alternating dimap} $ D^{g} $ is an alternating dimap $ D $ in which the face $ g $ is distinguished from the other faces. If $ D^{g} $ has genus zero, then the plane alternating dimap $ \hbox{Pl}(D^{g}) $ of $ D^{g} $ is obtained according to the process described above such that $ g $ is the outermost region of $ \hbox{Pl}(D^{g}) $. Conversely, given a plane alternating dimap $ P $ with outermost region $ g $, by reverse stereographic projection, we obtain a face-rooted alternating dimap $ D^{g} = \hbox{Pl}^{-1}(P) $ that has genus zero.

Suppose $ C $ is the closed trail forming the boundary of a face $ h $ in an alternating dimap $ D $, and $ H=D[C] $ (see Figure~\ref{fig:Induced_alternating_subdimap}). It is routine to show that $ H $ can be embedded on a sphere. The closed trail $ C $ can be partitioned into cycles, and each of these cycles encloses a face of opposite type to $ h $. For any face $ g \in F(H) \setminus \{h\} $, the face-rooted alternating dimap $ H^{g} $ can be used to obtain the plane alternating dimap $ \hbox{Pl}(H^{g}) $ (see Figure~\ref{fig:Face_rooted_alternating_dimap}).
\begin{figure}[t]
	\centering
	\includegraphics[width=0.3\textwidth]{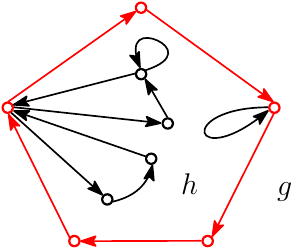}
	\caption{The plane alternating dimap $ \hbox{Pl}(H^{g}) $ of the face-rooted alternating dimap $ H^{g} $ where $ H $ is induced by the closed trail bounding $ h $}
	\label{fig:Face_rooted_alternating_dimap}
\end{figure}

In this scenario, the \emph{outer cycle} of $ h $ in $ D $ \emph{with respect to} $ g $ is the cycle formed by the common edges between $ h $ and $ g $ in $ H $. In Figure~\ref{fig:Face_rooted_alternating_dimap}, the outer cycle of $ h $ in $ D $ with respect to $ g $ is coloured red. When the choice of $ g $ is clear from the context, we may refer just to the \emph{outer cycle} of $ h $ in $ D $. If a plane alternating subdimap $ H = D[C] $ is converted into a plane alternating dimap $ P = \hbox{Pl}(H^{g}) $, we use the outermost region $ g $ of $ P $ to determine the outer cycle of the face bounded by $ C $ in $ D $.

An \emph{ordered alternating dimap}~\cite{Farr2018} is a pair $ (D,<) $ where $ D $ is an alternating dimap and $ < $ is a linear order on $ E(D) $. An ordered alternating dimap can be obtained by assigning a fixed edge-ordering to an alternating dimap.

Let $ \mathcal{G} $ be the set of plane graphs. Then, $ \hbox{alt}_c(\mathcal{G}) \coloneqq \{ \hbox{alt}_c(G) \mid G \in \mathcal{G} \} $ and $ \hbox{alt}_a(\mathcal{G}) \coloneqq \{ \hbox{alt}_a(G) \mid G \in \mathcal{G} \} $. 

\section{Triloops, Semiloops and Multiloops}\label{sec:type_of_loops}

There are a number of different types of special edges that have been defined in alternating dimaps including 1-loops, $ \omega $-loops and $ \omega^{2} $-loops~\cite{Farr2018}. An edge whose head has degree two is a \emph{1-loop}. A single edge forming an a-face is an \emph{$ \omega $-loop} whereas a single edge forming a c-face is an \emph{$ \omega^{2} $-loop}. An \emph{ultraloop} is concurrently a 1-loop, an $ \omega $-loop and an $ \omega^{2} $-loop. It is the only possible single-edge component in any alternating dimap. An illustration of these loops is given in Figure~\ref{fig:Loops}.
\begin{figure}[t]
	\centering
	\includegraphics[width=0.6\textwidth]{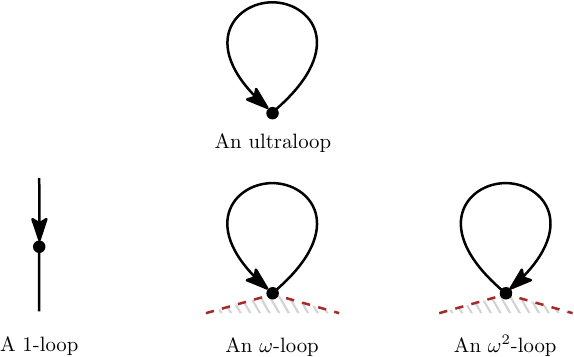}
	\caption{Loops}
	\label{fig:Loops}
\end{figure}

An edge is a \emph{triloop} if it is a 1-loop, an $ \omega $-loop or an $ \omega^{2} $-loop. In other words, it is a $ \mu $-loop for some $ \mu \in \{1, \omega, \omega^{2}\} $. If a $ \mu $-loop is not an ultraloop, then it is a \emph{proper $ \mu $-loop}. A proper $ \mu $-loop is a \emph{proper triloop}.

A \emph{1-semiloop} is a standard loop. We consider two scenarios in defining $ \omega $-semiloops and $ \omega^{2} $-semiloops. If a loop $ e $ is its own right successor, $ e $ is an \emph{$ \omega $-semiloop}. It is also an $ \omega^{2} $-loop under this circumstance. Note also that every $ \omega^{2} $-loop is an $ \omega $-semiloop. On the other hand, if $ e $ and its right successor are distinct, and they form a cutset of $ D $ or removal of them decreases the genus of $ D $, then $ e $ is also an \emph{$ \omega $-semiloop}. An $ \omega $-loop $ e $ is also an \emph{$ \omega^{2} $-semiloop} if $ e $ is its own left successor. Note also that every $ \omega $-loop is an $ \omega^{2} $-semiloop. If $ e $ and its left successor are distinct, and they form a cutset of $ D $ or removal of them decreases the genus of $ D $, then $ e $ is an \emph{$ \omega^{2} $-semiloop}. In both cases, by removing the cutsets, the number of components of $ D $ is increased, or the genus of $ D $ is decreased. For $ \mu \in \{1, \omega, \omega^{2}\} $, a $ \mu $-semiloop is a \emph{proper $ \mu $-semiloop} if it is not a triloop. An illustration of the three different types of semiloop is in Figure~\ref{fig:Semiloops}.
\begin{figure}[t]
	\centering
	\includegraphics[width=0.75\textwidth]{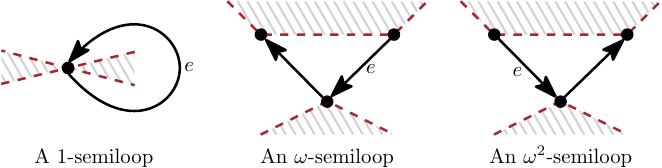}
	\caption{Semiloops}
	\label{fig:Semiloops}
\end{figure}

An edge is \emph{proper} if it is not a semiloop (and hence not a triloop or an ultraloop).

Suppose $ e $ and $ f $ are two loops in an alternating dimap $ D $. 
We can use our earlier definition (see page~\pageref{def:within}) of one block being within a face of another, and the fact that every loop is a block. The loop $ e $ is \emph{within} a face $ g $ of $ f $ if the point set formed by the embedding of $ e $ is a subset of the closure of $ g $.

Let $ D_{1} $ be an alternating dimap of genus zero with a single vertex $ v_{1} $ and $ \abs{E(D_{1})} = m \ge 1 $. Since every edge in $ D_{1} $ is a loop, there exists a clockwise face or an anticlockwise face of size one in $ D_{1} $. Suppose $ D_{1} $ has an $ \omega $-loop $ e $ and let $ f_{1} $ be the anticlockwise face of size one that is incident with $ e $. Let $ D_{2} $ be an alternating dimap of genus zero, $ f_{2} \in F(D_{2}) $ be an anticlockwise face and $ v_{2} \in V(\partial f_{2}) $. Suppose $ D = D_{1} \cup_{c} D_{2} $ with respect to $ f_{1}, v_{1}, f_{2} $ and $ v_{2} $. Then, $ D_{1} $ is a \emph{c-multiloop} of size $ m $ within the anticlockwise face $ f_{2} $ of $ D_{2} $, in $ D $. An example of a c-multiloop is shown in Figure~\ref{fig:c-multiloops}.
\begin{figure}[t]
	\centering
	\includegraphics[width=0.6\textwidth]{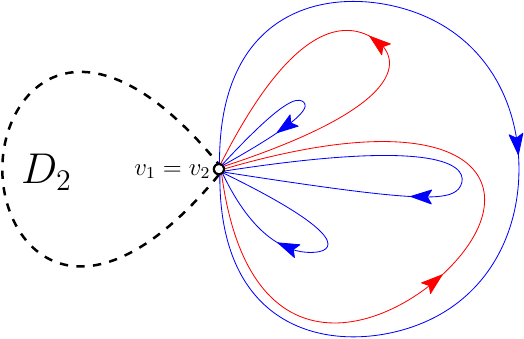}
	\caption{A c-multiloop $ D_{1} $ within an anticlockwise face of an alternating dimap $ D_{2} $ (the edges of the c-multiloop $ D_{1} $ are coloured red and blue, and the alternating dimap $ D_{2} $ is shown in dashed line)}
	\label{fig:c-multiloops}
\end{figure}
An \emph{a-multiloop} is defined by appropriate modifications.

A c-multiloop of size $ m $ within an anticlockwise face $ f_{1} $ of an alternating dimap $ D $ of genus zero can also be constructed as follows. Add in one proper $ \omega^{2} $-loop $ e $ that is incident with a vertex $ v \in V(\partial f_{1}) $. Denote by $ f_{2} $ the clockwise face of size one incident with $ e $. Then, $ m-1 $ proper $ \omega $-loops or $ \omega^{2} $-loops that are incident with $ v $ are added within the face $ f_{2} $ such that these $ m-1 $ loops only intersect at $ v $, and edges incident with $ v $ are directed inwards and outwards alternately in a cyclic order around $ v $. The $ m $ edges that are added into $ D $ form a c-multiloop within $ f_{1} $ of $ D $. Note that at the time each of the $ m-1 $ loops $ h $ is added within $ f_{2} $, if $ h $ is a proper $ \omega $-loop (respectively, $ \omega^{2} $-loop), it has an anticlockwise face (respectively, a clockwise face) of size one. If some other loops are added within the anticlockwise face (respectively, clockwise face) of size one of $ h $, then $ h $ is no longer a proper $ \omega $-loop (respectively, $ \omega^{2} $-loop).

\section{Minor Operations and Triality}\label{sec:reduction_operations}

In this section, we give the definitions of three minor operations for alternating dimaps, namely $ 1 $-reductions, $ \omega $-reductions and $ \omega^{2} $-reductions~\cite{Farr2018}.

Let $ D $ be an alternating dimap. Suppose $ u,v \in V(D) $ and $ e=uv \in E(D) $.

For $ \mu \in\{1,\omega,\omega^{2}\} $, the alternating dimap that is obtained by reducing $ e $ in $ D $ using $ \mu $-reduction is denoted by $ D[\mu]e $.

For the \emph{$ 1 $-reduction}, if $ e $ is an $ \omega $-loop or an $ \omega^{2} $-loop, the edge $ e $ is deleted to obtain $ D[1]e $. If $ e $ is not a loop, then $ D[1]e $ is obtained by contracting the edge $ e $. Note that contracting an edge in an alternating dimap always preserves the alternating property of the alternating dimap. If $ e $ is a $ 1 $-semiloop that is incident with a vertex $ v $, the alternating dimap $ D[1]e $ is formed as follows (see Figure~\ref{fig:1_Reduction_1semiloop}).
\begin{figure}[t]
	\centering
	\includegraphics[width=0.7\textwidth]{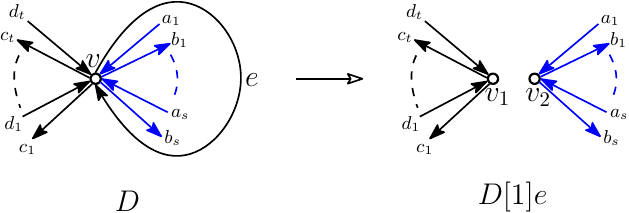}
	\caption{A 1-semiloop $ e $ in $ D $ is reduced by using the $ 1 $-reduction}
	\label{fig:1_Reduction_1semiloop}
\end{figure}
Let $ e,a_{1},b_{1},\ldots,a_{s},b_{s},e,c_{1},d_{1},\ldots,c_{t},d_{t} $ be the cyclic order of the edges that are incident with $ v $, starting from some edge $ e $ that is directed out from $ v $. Observe that each $ a_{i} $ and each $ d_{i} $ is an incoming edge of $ v $, and each $ b_{i} $ and each $ c_{i} $ is an outgoing edge of $ v $. In $ D[1]e $, the edge $ e $ is removed and the vertex $ v $ is split into two new vertices $ v_{1} $ and $ v_{2} $. Each $ a_{i} $ and each $ b_{i} $ is incident with $ v_{2} $, while each $ c_{i} $ and each $ d_{i} $ is incident with $ v_{1} $. The cyclic orderings of edges incident with $ v_{1} $ and $ v_{2} $ are induced by the cyclic ordering around $ v $. Note that this reduction may either increase the number of components or reduce the genus.

Let $ \ell=vn $ and $ r=vm $ be the left successor and the right successor of $ e=uv $ in $ D $, respectively.

For the \emph{$ \omega $-reduction}, if $ e $ is an $ \omega $-loop or an $ \omega^{2} $-loop, the edge $ e $ is deleted to obtain $ D[\omega]e $. Otherwise, the alternating dimap $ D[\omega]e $ is obtained by first deleting both of the edges $ e $ and $ \ell $, and a new edge $ \ell'=un $ is created such that the position of the tail (respectively, head) of $ \ell' $ in the cyclic ordering of edges incident with $ u $ (respectively, $ n $) in $ D[\omega]e $ is the same as the position of the tail (respectively, head) of $ e $ (respectively, $ \ell $) in the cyclic ordering of edges incident with $ u $ (respectively, $ n $) in $ D $. If $ \hbox{deg}(v)=2 $, the vertex $ v $ is also removed from $ D[\omega]e $. Note that if $ e $ is a proper $ \omega^{2} $-semiloop then this may increase the number of components or reduce the genus.

For the \emph{$ \omega^{2} $-reduction}, if $ e $ is an $ \omega $-loop or an $ \omega^{2} $-loop, the edge $ e $ is deleted to obtain $ D[\omega^{2}]e $. Otherwise, the alternating dimap $ D[\omega^{2}]e $ is obtained by first deleting both of the edges $ e $ and $ r $, and a new edge $ r'=um $ is created such that the position of the tail (respectively, head) of $ r' $ in the cyclic ordering of edges incident with $ u $ (respectively, $ m $) in $ D[\omega^{2}]e $ is the same as the position of the tail (respectively, head) of $ e $ (respectively, $ r $) in the cyclic ordering of edges incident with $ u $ (respectively, $ m $) in $ D $. If $ \hbox{deg}(v)=2 $, the vertex $ v $ is also removed from $ D[\omega^{2}]e $. Note that if $ e $ is a proper $ \omega $-semiloop then this may increase the number of components or reduce the genus.

We call these three operations the \emph{reduction operations} or the \emph{minor operations} for alternating dimaps.

A \textit{minor} of an alternating dimap $ D $ is obtained by reducing some of its edges using a sequence of reduction operations.

For the reduction of a triloop $ e \in E(D) $, we have $ D[1]e=D[\omega]e=D[\omega^{2}]e $. Since the type of reduction operation is insignificant, we sometimes write $ D[*]e $ when a triloop $ e $ is reduced.

For $ \mu \in \{1,\omega,\omega^{2}\} $, if $(D,<)$ is an ordered alternating dimap, then the $ \mu $\emph{-reduction} $ (D,<)[\mu] $ of $ (D,<) $ is the ordered alternating dimap $ (D[\mu]e_{0},<') $ where $ e_{0} $ is the first edge in $ E(D) $ under $ < $ and the order $ <' $ on $ E(D)\setminus \{e_{0}\} $ is obtained by simply removing $ e_{0} $ from the order $ < $.

The concept of \emph{triality} (or \emph{trinity}) was introduced by Tutte when he studied the dissections of equilateral triangles~\cite{Tutte1948}. See, for example~\cite{Berman1980,Farr2018,Tutte1948,Tutte1973} for full details. We use the notation of~\cite{Farr2018}. The trial of alternating dimap $ D $ is denoted by $ D^{\omega} $, and the trial of $ D^{\omega} $ is $ D^{\omega^{2}} $. We put $ \omega^{3} = 1 $ and $ D^{1} = D $, so that the trial of $ D^{\omega^{2}} $ is $ D^{\omega^{3}} = D^{1} = D $, as required by triality. The images of an edge $ e \in E(D) $ under successive trials are denoted by $ e^{\omega} \in E(D^{\omega}) $ and $ e^{\omega^{2}} \in E(D^{\omega^{2}}) $, with $ e^{\omega^{3}} = e^{1} = e $. The edge types of $ e^{\omega} $ and $ e^{\omega^{2}} $ are shown in Table~\ref{tab:trial_edge_type}.
\begin{table}[ht]
	\centering
	\begin{tabular}{|c|c|c|}
		\hline
		$ e $ & $ e^{\omega} $ & $ e^{\omega^{2}} $\\
		\hline \hline
			ultraloop & ultraloop & ultraloop\\
		\hline
			proper $ 1 $-loop & proper $ \omega $-loop & proper $ \omega^{2} $-loop\\
		\hline
			proper $ 1 $-semiloop & proper $ \omega $-semiloop & proper $ \omega^{2} $-semiloop\\
		\hline
			proper edge & proper edge & proper edge\\
		\hline
	\end{tabular}
	\caption{The edge type of an edge $ e $ after trial operations on $ e $ (see Figures~\ref{fig:Loops} and \ref{fig:Semiloops} for different edge types)}
	\label{tab:trial_edge_type}
\end{table}

\section{Characterisations of Extended Tutte Invariants}\label{sec:characterisations_of_ETI}

We characterise extended Tutte invariants in this section. The definition of extended Tutte invariants is given in Definition~\ref{def:Extended_Tutte_Invariant}. These are intended to be as general as possible, in that they use different linear recursions for all the different edge types, with no assumptions made about relations between the factors introduced for each reduction. For brevity, we use $ P(D) $ as a shorthand for \[ P(D;w,x,y,z,a,b,c,d,e,f,g,h,i,j,k,l) \] throughout this article.

Throughout, $ \mathcal{A} $ denotes a class of alternating dimaps.

\begin{definition}\label{def:Multiplicative Invariant}
	Let $ \mathbb{F} $ be a field. A \emph{multiplicative invariant (over $ \mathbb{F} $)} for alternating dimaps in $ \mathcal{A} $ is a function $ P \colon \mathcal{A} \to \mathbb{F} $, such that $ P $ is invariant under isomorphism, $ P(\emptyset)=1 $ and for the disjoint union of two alternating dimaps, $ G $ and $ H $, $ P(G\cup H)=P(G) \cdot P(H) $.
\end{definition}

\begin{definition}\label{def:Extended_Tutte_Invariant}
	An \emph{extended Tutte invariant} for alternating dimaps in $ \mathcal{A} $ with respect to a parameter sequence $ (w,x,y,z,a,b,c,d,e,f,g,h,i,j,k,l) $ of elements of a field $ \mathbb{F} $ is a multiplicative invariant $ P $ over $ \mathbb{F} $ such that for any alternating dimap $ D \in \mathcal{A} $ and $ r \in E(D) $,
	\begin{enumerate}
		\item if $ r $ is an ultraloop,
		\begin{equation*}
			P(D)=w\cdot P(D\setminus r), \tag{ETI1}
		\end{equation*}
		
		\item if $ r $ is a proper 1-loop,
		\begin{equation*}
			P(D)=x\cdot P(D[1]r), \tag{ETI2}
		\end{equation*}
		
		\item if $ r $ is a proper $ \omega $-loop,
		\begin{equation*}
			P(D)=y\cdot P(D[\omega]r), \tag{ETI3}
		\end{equation*}
		
		\item if $ r $ is a proper $ \omega^{2} $-loop,
		\begin{equation*}
			P(D)=z\cdot P(D[\omega^{2}]r), \tag{ETI4}
		\end{equation*}
		
		\item if $ r $ is a proper 1-semiloop,
		\begin{equation*}
			P(D)=a\cdot P(D[1]r)+b\cdot P(D[\omega]r)+c\cdot P(D[\omega^{2}]r), \tag{ETI5}
		\end{equation*}
		
		\item if $ r $ is a proper $ \omega $-semiloop,
		\begin{equation*}
			P(D)=d\cdot P(D[1]r)+e\cdot P(D[\omega]r)+f\cdot P(D[\omega^{2}]r), \tag{ETI6}
		\end{equation*}
		
		\item if $ r $ is a proper $ \omega^{2} $-semiloop,
		\begin{equation*}
			P(D)=g\cdot P(D[1]r)+h\cdot P(D[\omega]r)+i\cdot P(D[\omega^{2}]r), \tag{ETI7}
		\end{equation*}
		
		\item otherwise,
		\begin{equation*}
			P(D)=j\cdot P(D[1]r)+k\cdot P(D[\omega]r)+l\cdot P(D[\omega^{2}]r). \tag{ETI8}
		\end{equation*}
	\end{enumerate}
\end{definition}

To define extended Tutte invariants for ordered alternating dimaps $ (D,<) $, we have the following modifications:
\begin{enumerate}
	\item For $ \mu \in \{1,\omega,\omega^{2}\} $, each $ \mu $-reduction is replaced by $ (D,<)[\mu] $.
	\item The edge to be reduced is always the first edge $ e_{0} $ in the linear order $ < $ on $ E(D) $, so the reference to edges is omitted for each reduction operation.
\end{enumerate}

\subsection{Arbitrary Alternating Dimaps, Dependent Parameters}

In this section, we require $ \mathcal{A} $ to be the set of all alternating dimaps of genus zero. In \cref{sec:well_defined_ETI}, we will consider extended Tutte invariants that are only well defined for certain alternating dimaps. 

It is well known that if a graph $ G $ is planar and $ G^{*} $ is the dual graph of $ G $, then 
\begin{equation*}
	T(G;x,y) = T(G^{*};y,x).
\end{equation*}

We give an analogous relation for extended Tutte invariants in Theorem~\ref{thm:trial_ETI}, by using the following theorem by Farr:
\begin{theorem}[{\cite[Theorem 2.2]{Farr2018}}]\label{thm:triality_minor}
	If $ e \in E(D) $ and $ \mu,\nu \in \{1,\omega,\omega^{2}\} $ then \[ D^{\mu}[\nu]e^{\mu}=(D[\mu\nu]e)^{\mu}. \]
\end{theorem}

\begin{theorem}\label{thm:trial_ETI}
	For any extended Tutte invariant $ P $ of an alternating dimap $ D $,
		\begin{align*}
			& P(D;w,x,y,z,a,b,c,d,e,f,g,h,i,j,k,l)\\
			& = P(D^{\omega};w,z,x,y,h,i,g,b,c,a,e,f,d,k,l,j)\\
			& = P(D^{\omega^{2}};w,y,z,x,f,d,e,i,g,h,c,a,b,l,j,k).
		\end{align*}	
\end{theorem}

\begin{proof}
	Induction on $ \abs{E(D)} $. See~\cite{YowPHD2019} for details.
\end{proof}

We proceed to characterise the extended Tutte invariant. We will use the following simple observation.
\begin{lemma}\label{lem:triloop}
	Let $ D $ be an alternating dimap of genus zero that has size $ m $ and $ 0 \le r \le m $. If $ r=2 $ or $ 3 $, the reduced alternating dimap $ D[\mu_{1}]e_{1}[\mu_{2}]e_{2} \ldots [\mu_{m-r}]e_{m-r} $ has a triloop.
\end{lemma}

\begin{proof}
	Every alternating dimap $ D $ of genus zero that has size two or three contains a triloop. There are $ r $ edges remaining in the reduced alternating dimap after $ m-r $ reductions.
\end{proof}

We extend this result and prove that some sequence of reductions on a connected alternating dimap of genus zero gives a proper triloop in a reduced alternating dimap that has size at least three. Observe that not all alternating dimaps of genus zero that have size three contain a proper triloop.
\begin{lemma}\label{lem:proper_triloop}
	If $ D $ is a connected alternating dimap of genus zero that has size at least three, then some minor of $ D $ with at least three edges contains a proper triloop.
\end{lemma}

\begin{proof}
	\begin{figure}[H]
		\centering
		\includegraphics[width=0.8\textwidth]{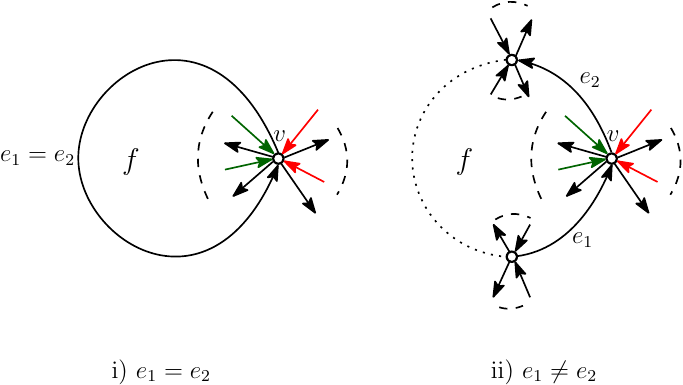}
		\caption{The anticlockwise face $ f $ that has $ C $ as its outer cycle in the proof of Lemma~\ref{lem:proper_triloop}}
		\label{fig:Minor_proper_triloops}
	\end{figure}

	Let $ D $ be as stated. By inspection, $ \abs{E(D)}=3 $	implies that $ D $ contains a proper triloop, so no reduction operation is needed in this case. Now, suppose $ \abs{E(D)}>3 $. Assume there is no proper triloop in $ D $ (else, again, no reduction is needed, and $D$ itself is the required minor).
	
	Every non-empty alternating dimap contains at least two faces. Pick an arbitrary closed trail $ C $ that forms an anticlockwise face $ f $ (or a clockwise face with appropriate modifications in the following steps) in $ D $. Let $ H=D[C] $. Suppose $ R $ is the outer cycle of $ f $ and $ v \in V(R) $. Let $ e_{1},e_{2} \in E(R) $ be the edges that are directed into and out from the vertex $ v $, respectively, and they partition $ T=I(v)\setminus e_{1} $ in $ D $ into two sets  (based on the cyclic order of $ T $), $ (i) $ $ S_{c} $ that contains every edge directed into $ v $ that lies between $ e_{1} $ and $ e_{2} $ as we go from $ e_{1} $ to $ e_{2} $ in clockwise order around $ v $, and $ (ii) $ $ S_{a}=T\setminus S_{c} $ (see Figure~\ref{fig:Minor_proper_triloops}, where edges in $ S_{c} $ and $ S_{a} $ are highlighted in green and red, respectively).
	
	First, suppose $ e_{1}=e_{2} $ (as shown in Figure~\ref{fig:Minor_proper_triloops}(i)). Since there exists no proper triloop in $ D $, both the sets $ S_{c} $ and $ S_{a} $ are not empty. ($ S_{c} $ and $ S_{a} $ being non-empty implies that each of them has size of at least two, by the definition of alternating dimaps and using the fact that there exists no proper triloop in $ D $.) By reducing every edge in $ S_{c} $ by $ \omega^{2} $-reductions, $ e_{1} $ is now a proper $ \omega $-loop and there are at least three edges in the reduced alternating dimap.
	
	Second, suppose $ e_{1} \ne e_{2} $ (as shown in Figure~\ref{fig:Minor_proper_triloops}(ii)). By performing $ \omega^{2} $-reductions on every edge in $ S_{c} $, and $ \omega $-reductions on every edge in $ S_{a} $, the edge $ e_{1} $ is then a proper $ 1 $-loop. Since there exists no proper triloop in $ D $, the edge $ e_{2} $ is not a proper $ 1 $-loop. Hence, there are at least three edges in the reduced alternating dimap.

	Therefore, the result follows.
\end{proof}

A \emph{derived polynomial} for an alternating dimap $ D $ is a polynomial in variables $ w,x,y,z,a,b,c,d,e,f,g,h,i,j,k,l $ obtained as an extended Tutte invariant for $ (D,<) $ where $ < $ is a fixed edge-ordering on $ E(D) $. The $ m! $ permutations of the edge set of an alternating dimap of size $ m $ give $ m! $ derived polynomials, where some of them may be identical.

We write $ G_{n,m} $ for an alternating dimap $ G $ that consists of $ n $ vertices and $ m $ edges such that there exists at least one edge that is not a triloop.

Since there are two non-isomorphic alternating dimaps that may be denoted by $ G_{2,3} $, we write $ G^{a}_{2,3} $ and $ G^{c}_{2,3} $ for the alternating dimap $ G_{2,3} $ that contains one anticlockwise face of size three and one clockwise face of size three, respectively. The possibilities for $ G_{1,3} $ and $ G_{2,3} $ are shown in Figure~\ref{fig:G_{n,m}}.
\vspace{-1cm}
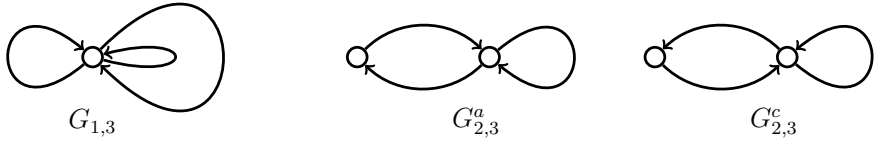
\begin{figure}[H]
	\centering
	\resizebox{0.8\textwidth}{!}
	{%
		\begin{tikzpicture}
		[every path/.style={color=black, line width=1.2pt}, 
		every node/.style={draw, circle, line width=1.2pt, inner sep=3pt},
		bend angle=45]
		
		\node	(0) at (0,0)	{};
		\path	(0)	edge[->, in=140, out=220, looseness=2, min distance=20mm]	(0);
		\path	(0)	edge[->, in=15, out=345, looseness=1, min distance=15mm]	(0);
		\path	(0)	edge[->, in=315, out=45, looseness=2, min distance=35mm]	(0);
		
		\node	(1) at (4,0)	{};
		\node	(2) at (6,0)	{}
			edge[->, bend left]		(1)
			edge[<-, bend right]	(1);
		\path	(2)	edge[->, in=320, out=40, looseness=4, min distance=20mm]	(2);
		
		\node	(3) at (8.5,0)	{};
		\node	(4) at (10.5,0)	{}
			edge[->, bend right]	(3)
			edge[<-, bend left]		(3);
		\path	(4)	edge[<-, in=320, out=40, looseness=4, min distance=20mm]	(4);
		
		\node[draw=none]	(3) at (0,-1)	{$ G_{1,3} $};
		\node[draw=none]	(3) at (5.8,-1)	{$ G^{a}_{2,3} $};
		\node[draw=none]	(4) at (10.3,-1)	{$ G^{c}_{2,3} $};
		\end{tikzpicture}
	}%
	\vspace{-10mm}
	\caption{Alternating dimaps $ G_{1,3} $ and $ G_{2,3} $}
	\label{fig:G_{n,m}}
\end{figure}

We give the derived polynomials for an alternating dimap $ G_{2,4} $ in the following lemma.
\begin{lemma}\label{lem:derived_polynomials_G_{2,4}}
	Let $ P $ be an extended Tutte invariant. There exist exactly 12 distinct derived polynomials for the alternating dimap $ G_{2,4} $ as shown in Figure~\ref{fig:G_{2,4}} , namely
	\begin{enumerate}[listparindent=\parindent]
		\item[E1.]
		$ P(D) = ( jwyz $ \emph{or} $ j(aww+bwy+cwz) )$\\
		\indent \indent $ \! \! + \, ( kwxy $ \emph{or} $ k(gwy+hww+iwx) )$\\
		\indent \indent	$ \! \! + \, ( lwxz $ \emph{or} $ l(dwz+ewx+fww) )$,
		\item[E2.] $ P(D)=wxyz $ or $ awwx+bwxy+cwxz $,
		\item[E3.] $ P(D)=wxyz $ or $ dwyz+ewxy+fwwy $,
		\item[E4.] $ P(D)=wxyz $ or $ gwyz+hwwz+iwxz $.
	\end{enumerate}
\end{lemma}

\begin{figure}[ht]
	\centering
	\resizebox{0.5\textwidth}{!}
	{%
		\begin{tikzpicture}
		[every path/.style={color=black, line width=1.2pt}, 
		every node/.style={draw, circle, line width=1.2pt, inner sep=3pt},
		bend angle=35]
		
		\node	(u) at (0,0)	{$ u $};
		\node	(v) at (6,0)	{$ v $}
			edge[->, bend left]	node[draw=none, below]	{$ n $}	(u)
			edge[<-, bend right] node[draw=none, above]	{$ o $}	(u);
		\path	(u)	edge[->, in=140, out=220, looseness=2, min distance=20mm] node[draw=none, left] {$ q $}	(u);
		\path	(u)	edge[->, in=15, out=345, looseness=0, min distance=40mm] node[draw=none, right] {$ p $}	(u); 
		\end{tikzpicture}
	}%
	\caption{The alternating dimap $ G_{2,4} $}
	\label{fig:G_{2,4}}
\end{figure}
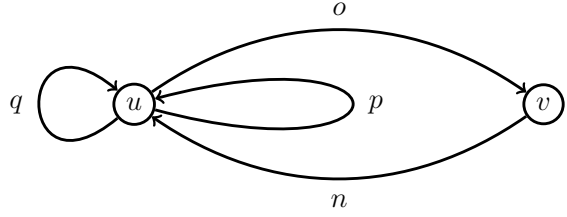
\noindent
\textbf{Remark:} We use ``or"  in (E1)--(E4) to show all the possible derived polynomials when the respective edge is first reduced. For instance, (E1) gives eight derived polynomials in total.

\begin{proof}
	Since the alternating dimap $ D \cong G_{2,4} $ has four edges, there exist $ 4!=24 $ possible edge-orderings. To obtain the derived polynomials as in (E1)--(E4), the first edge to be reduced in $ D $ is $ n,o,p $ and $ q $, respectively.
\end{proof}

\begin{lemma}\label{lem:derived_polynomials_G_{1,3}_G^{a}_{2,3}_G^{c}_{2,3}}
	Let $ P $ be an extended Tutte invariant.
	\begin{enumerate}
		\item[a)] The only distinct derived polynomials for the alternating dimap $ G_{1,3} $ are
		\begin{equation*}
			P(D)=wyz \quad \text{and} \quad P(D)=aww+bwy+cwz.
		\end{equation*}
		
		\item[b)] The only distinct derived polynomials for the alternating dimap $ G^{a}_{2,3} $ are
		\begin{equation*}
			P(D)=wxz \quad \text{and} \quad P(D)=dwz+ewx+fww.
		\end{equation*}
		
		\item[c)] The only distinct derived polynomials for the alternating dimap $ G^{c}_{2,3} $ are
		\begin{equation*}
			P(D)=wxy \quad \text{and} \quad P(D)=gwy+hww+iwx.
		\end{equation*}
	\end{enumerate}
\end{lemma}

\begin{proof}
	Direct calculation.
\end{proof}

Since the final edge to be reduced in each component of any non-empty alternating dimap $ D $ is always an ultraloop, we have $ P(D) \equiv 0 $ if $ w=0 $. Hence, we assume that $ w \ne 0 $ hereinafter.

As shown in Lemma~\ref{lem:proper_triloop}, every connected alternating dimap $ D $ of genus zero that has size at least three has a minor that contains a proper triloop. By Definition~\ref{def:Extended_Tutte_Invariant}, one variable $ x,y $ or $ z $ is produced if the proper triloop in $ D $ is reduced first. If that variable equals zero, a trivial solution will then be obtained. Hence, we first characterise extended Tutte invariants for alternating dimaps of genus zero under the assumption that $ x,y,z \ne 0 $.

\begin{theorem}\label{thm:Extended_Tutte_Invariant}
	Let $ S=(w,x,y,z,a,b,c,d,e,f,g,h,i,j,k,l) $ be a parameter sequence such that $ w,x,y,z \ne 0 $. A function $ P $ is an extended Tutte invariant with respect to $ S $ for every alternating dimap $ D $ of genus zero if and only if
	\begin{equation}\label{eqn:Extended_Tutte_Invariant}
		P(D)=w^{k(D)} \cdot x^{\instar(D)-k(D)} \cdot y^{\aface(D)-k(D)} \cdot z^{\cface(D)-k(D)}
	\end{equation}
	with
	\begin{equation}\label{eqn:necessary_conditions_for_ETI_1}
		xyz=jyz+kxy+lxz,
	\end{equation}	
	\begin{equation}\label{eqn:necessary_conditions_for_ETI_2}
		yz=aw+by+cz,
	\end{equation}	
	\begin{equation}\label{eqn:necessary_conditions_for_ETI_3}
		xz=dz+ex+fw,
	\end{equation}	
	\begin{equation}\label{eqn:necessary_conditions_for_ETI_4}
		xy=gy+hw+ix.
	\end{equation}
\end{theorem}

\begin{proof}
	Let $ S $ be as stated and $ D $ be an alternating dimap of genus zero. We first prove the forward implication. Note that all the derived polynomials must be equal for $ P(D) $ to be an extended Tutte invariant. By (E1)--(E4) in Lemma~\ref{lem:derived_polynomials_G_{2,4}} and as $ w,x,y,z \ne 0 $, we obtain (\ref{eqn:necessary_conditions_for_ETI_1})--(\ref{eqn:necessary_conditions_for_ETI_4}) as desired. For example, by equating two of the 12 derived polynomials $ jwyz + kwxy + lwxz $ (from E1) and $ wxyz $ (from E2) in Lemma~\ref{lem:derived_polynomials_G_{2,4}}, and divide both sides of the equation by $ w \ne 0 $, we obtain (\ref{eqn:necessary_conditions_for_ETI_1}).
	
	We next show \[ P(D)=w^{k(D)} \cdot x^{\instar(D)-k(D)} \cdot y^{\aface(D)-k(D)} \cdot z^{\cface(D)-k(D)}, \] using induction on $ \abs{E(D)}=m $. There exist eight cases corresponding to the eight categories ((\hyperref[def:Extended_Tutte_Invariant]{ETI1}) to (\hyperref[def:Extended_Tutte_Invariant]{ETI8})) in Definition~\ref{def:Extended_Tutte_Invariant}. For the base case, suppose $ m=0 $. Clearly, $ P(D)=1 $ and the result follows. Assume that $ m>0 $ and the result holds for every alternating dimap of genus zero that has size less than $ m $. Let $ r \in E(D) $.
	\begin{enumerate}
		\item[i)] \textbf{$ r $ is an ultraloop}. In $ D\setminus r $, the number of components, in-stars, a-faces and c-faces are all reduced by 1. Thus,
		\begin{align*}
		P(D) & = w \cdot P(D\setminus r)\\
			& = w \cdot w^{k(D\setminus r)} \cdot 
			x^{\instar(D\setminus r)-k(D\setminus r)} \cdot 
			y^{\aface(D\setminus r) -k(D\setminus r)} \cdot 
			z^{\cface(D\setminus r)-k(D\setminus r)} \tag{by the inductive hypothesis}\\
			& = w \cdot w^{k(D)-1} \cdot 
			x^{\instar(D)-1-(k(D)-1)} \cdot 
			y^{\aface(D)-1-(k(D)-1)} \cdot 
			z^{\cface(D)-1-(k(D)-1)}\\
			& = w^{k(D)} \cdot x^{\instar(D)-k(D)} \cdot y^{\aface(D)-k(D)} \cdot z^{\cface(D)-k(D)}.
		\end{align*}
		
		\item[ii)] \textbf{$ r $ is a proper $ 1 $-loop}. The number of in-stars is reduced by 1 in $ D[1]r $. Then,
		\begin{align*}
		P(D) & = x \cdot P(D[1]r)\\
			& = x \cdot w^{k(D[1]r)} \cdot 
			x^{\instar(D[1]r)-k(D[1]r)} \cdot 
			y^{\aface(D[1]r)-k(D[1]r)} \cdot 
			z^{\cface(D[1]r)-k(D[1]r)} \tag{by the inductive hypothesis}\\
			& = x \cdot w^{k(D)} \cdot 
			x^{\instar(D)-1-k(D)} \cdot 
			y^{\aface(D)-k(D)} \cdot 
			z^{\cface(D)-k(D)}\\
			& = w^{k(D)} \cdot x^{\instar(D)-k(D)} \cdot y^{\aface(D)-k(D)} \cdot z^{\cface(D)-k(D)}.
		\end{align*}

		\item[iii)] \textbf{$ r $ is a proper $ 1 $-semiloop}. In $ D[1]r $, the number of components and in-stars are both increased by 1. In $ D[\omega]r $ and $ D[\omega^{2}]r $, the number of c-faces and a-faces are reduced by 1, respectively. Hence,
		\begin{align*}
		yz \cdot P(D) & = yz \cdot \Big( a \cdot P(D[1]r) + b \cdot P(D[\omega]r) + c \cdot P(D[\omega^{2}]r) \Big)\\
			& = yz \cdot \Big( a \cdot w^{k(D[1]r)} \cdot 
			x^{\instar(D[1]r)-k(D[1]r)} \cdot 
			y^{\aface(D[1]r)-k(D[1]r)} \cdot 
			z^{\cface(D[1]r)-k(D[1]r)}\\
			& \quad + b \cdot w^{k(D[\omega]r)} \cdot 
			x^{\instar(D[\omega]r)-k(D[\omega]r)} \cdot 
			y^{\aface(D[\omega]r)-k(D[\omega]r)} \cdot 
			z^{\cface(D[\omega]r)-k(D[\omega]r)}\\
			& \quad + c \cdot w^{k(D[\omega^{2}]r)} \cdot 
			x^{\instar(D[\omega^{2}]r)-k(D[\omega^{2}]r)} \cdot 
			y^{\aface(D[\omega^{2}]r)-k(D[\omega^{2}]r)} \cdot 
			z^{\cface(D[\omega^{2}]r)-k(D[\omega^{2}]r)} \Big) \tag{by the inductive hypothesis}\\
			& = yz \cdot \Big( a \cdot w^{k(D)+1} \cdot x^{\instar(D)+1-(k(D)+1)} \cdot y^{\aface(D)-(k(D)+1)} \cdot z^{\cface(D)-(k(D)+1)}\\
			& \quad + b \cdot w^{k(D)} \cdot x^{\instar(D)-k(D)} \cdot y^{\aface(D)-k(D)} \cdot z^{\cface(D)-1-k(D)}\\
			& \quad + c \cdot w^{k(D)} \cdot x^{\instar(D)-k(D)} \cdot y^{\aface(D)-1-k(D)} \cdot z^{\cface(D)-k(D)} \Big)\\
			& = \left( aw + by + cz \right) \cdot 
			w^{k(D)} \cdot x^{\instar(D)-k(D)} \cdot y^{\aface(D)-k(D)} \cdot z^{\cface(D)-k(D)}\\
		P(D) & = w^{k(D)} \cdot x^{\instar(D)-k(D)} \cdot y^{\aface(D)-k(D)} \cdot z^{\cface(D)-k(D)} \tag[.]{by (\ref{eqn:necessary_conditions_for_ETI_2})}
		\end{align*}

		\item[iv)] \textbf{$ r $ is a proper edge}. Observe that the number of in-stars, c-faces and a-faces, are all reduced by 1 in $ D[1]r $, $ D[\omega]r $ and $ D[\omega^{2}]r $, respectively. Therefore, we have,
		\begin{align*}
		xyz \cdot P(D) & = xyz \cdot \Big( j \cdot P(D[1]r) + k \cdot P(D[\omega]r) + l \cdot P(D[\omega^{2}]r) \Big)\\
			& = xyz \cdot \Big( j \cdot w^{k(D[1]r)} \cdot 
			x^{\instar(D[1]r)-k(D[1]r)} \cdot 
			y^{\aface(D[1]r)-k(D[1]r)} \cdot 
			z^{\cface(D[1]r)-k(D[1]r)}\\
			& \quad + k \cdot w^{k(D[\omega]r)} \cdot 
			x^{\instar(D[\omega]r)-k(D[\omega]r)} \cdot 
			y^{\aface(D[\omega]r)-k(D[\omega]r)} \cdot 
			z^{\cface(D[\omega]r)-k(D[\omega]r)}\\
			& \quad + l \cdot w^{k(D[\omega^{2}]r)} \cdot 
			x^{\instar(D[\omega^{2}]r)-k(D[\omega^{2}]r)} \cdot 
			y^{\aface(D[\omega^{2}]r)-k(D[\omega^{2}]r)} \cdot 
			z^{\cface(D[\omega^{2}]r)-k(D[\omega^{2}]r)} \Big) \tag{by the inductive hypothesis}\\
			& = xyz \cdot \Big( j \cdot w^{k(D)} \cdot x^{\instar(D)-1-k(D)} \cdot y^{\aface(D)-k(D)} \cdot z^{\cface(D)-k(D)}\\
			& \quad + k \cdot w^{k(D)} \cdot x^{\instar(D)-k(D)} \cdot y^{\aface(D)-k(D)} \cdot z^{\cface(D)-1-k(D)}\\
			& \quad + l \cdot w^{k(D)} \cdot x^{\instar(D)-k(D)} \cdot y^{\aface(D)-1-k(D)} \cdot z^{\cface(D)-k(D)} \Big)\\
			& = \left( jyz+kxy+lxz \right) \cdot
			w^{k(D)} \cdot x^{\instar(D)-k(D)} \cdot y^{\aface(D)-k(D)} \cdot z^{\cface(D)-k(D)}\\
		P(D) & = w^{k(D)} \cdot x^{\instar(D)-k(D)} \cdot y^{\aface(D)-k(D)} \cdot z^{\cface(D)-k(D)} \tag[.]{by (\ref{eqn:necessary_conditions_for_ETI_1})}
		\end{align*}
	\end{enumerate}
	
	The arguments for the other four cases (proper $ \omega $-loop/$ \omega^{2} $-loop/$ \omega $-semiloop/$ \omega^{2} $-semiloop) are similar, and are given in full in~\cite{YowPHD2019}.
	
	Conversely, it must be shown that \[ P(D)=w^{k(D)} \cdot x^{\instar(D)-k(D)} \cdot y^{\aface(D)-k(D)} \cdot z^{\cface(D)-k(D)} \] with (\ref{eqn:necessary_conditions_for_ETI_1})--(\ref{eqn:necessary_conditions_for_ETI_4}) is an extended Tutte invariant with respect to $ S $ for every alternating dimap $ D $, as in Definition~\ref{def:Extended_Tutte_Invariant}. It is routine to show that $ P(D) $ is a multiplicative invariant.

	The proof that $ P(D) $ satisfies (\hyperref[def:Extended_Tutte_Invariant]{ETI1}) to (\hyperref[def:Extended_Tutte_Invariant]{ETI8}) uses induction on $ \abs{E(D)}=m $. The approach is similar in style to the above induction. See~\cite{YowPHD2019} for details.
\end{proof}

As seen in the proof of Theorem~\ref{thm:Extended_Tutte_Invariant}, variables $ x,y $ and $ z $ must be non-zero in order to complete the characterisation. We next consider cases where at least one of these three variables is zero, using different arguments. We shall first establish some excluded minor characterisations of alternating dimaps of genus zero.

\begin{lemma}\label{lem:bigger_c-face_to_alt_c}
	Let $ D $ be an alternating dimap. If every clockwise face of $ D $ has size at least two, then $ D $ can be reduced to an alternating dimap that contains $ \cface(D) $ clockwise faces of size exactly two, using a sequence of contraction operations.
\end{lemma}

\begin{proof}		
	Let $ D $ be as stated. We proceed by induction on $ \abs{E(D)}=m $. For the base case, suppose $ m=2 $. Since there is a clockwise face of size exactly two, the result for $ m=2 $ follows.
	
	For the inductive step, assume that $ m>2 $ and the result holds for every alternating dimap of size less than $ m $.
	
	Suppose $ k=\cface(D) $, and $ g \in F(D) $ is a clockwise face of size greater than two. Let $ e \in E(\partial g) $. By contracting the edge $ e $, the size of $ g $ will be reduced by one. As every edge in an alternating dimap belongs to one clockwise face and one anticlockwise face, we have $ \cface(D/e)=k $.  By the inductive hypothesis, the alternating dimap $ D/e $ can be reduced to an alternating dimap that contains $ k $ clockwise faces of size exactly two, using a sequence of contraction operations. Therefore, alternating dimap $ D $ can be reduced to an alternating dimap that contains $ k $ clockwise faces of size exactly two by a sequence of contraction operations, namely, contraction of $ e $ followed by the aforementioned contraction sequence for $ D/e $.
\end{proof}

\begin{lemma}\label{lem:bigger_a-face_to_alt_a}
	Let $ D $ be an alternating dimap. If every anticlockwise face of $ D $ has size at least two, then $ D $ can be reduced to an alternating dimap that contains $ \aface(D) $ anticlockwise faces of size exactly two, using a sequence of contraction operations.
\end{lemma}

\begin{proof}
	The result follows by some appropriate modifications to the proof of Lemma~\ref{lem:bigger_c-face_to_alt_c}.
\end{proof}

We show that $ G_{1,3} $, $ G^{a}_{2,3} $ or $ G^{c}_{2,3} $ (see Figure~\ref{fig:G_{n,m}}) is a minor for certain alternating dimaps, in the following lemmas.
\begin{lemma}\label{lem:minors_G_13}
	Every alternating dimap of genus zero that contains a proper $ 1 $-semiloop has $ G_{1,3} $ as a minor.
\end{lemma}

\begin{proof}
	Let $ D $ be an alternating dimap of genus zero that contains a proper $ 1 $-semiloop $ e $. We proceed by induction on $ \abs{V(D)}=n $. For the base case, suppose $ n=1 $. The alternating dimap $ G_{1,3} $ can be obtained by repeatedly reducing some proper triloops.
	
	For the inductive step, assume that $ n>1 $ and the result holds for every $ D $ that has less than $ n $ vertices. If $ e $ belongs to a component that has exactly one vertex in $ D $, from the base case, the alternating dimap $ D $ contains $ G_{1,3} $ as a minor. So, suppose $ e $ belongs to a component $ P $ that has at least two vertices in $ D $. This implies that there exists at least one non-loop edge $ f $ in $ P $. By contracting $ f $, we have $ \abs{V(D/f)}=n-1 $. By the inductive hypothesis, the alternating dimap $ D/f $ contains $ G_{1,3} $ as a minor. Since $ D/f $ is a minor of $ D $, the result follows.
\end{proof}

Since $ G_{1,3} $ contains a proper $ \omega $-loop and a proper $ \omega^{2} $-loop, the following corollary follows from Definition~\ref{def:Extended_Tutte_Invariant}.
\begin{corollary}\label{cor:1-semiloop}
	If $ y=0 $ or $ z=0 $, and there exists a proper $ 1 $-semiloop in an alternating dimap $ D $ of genus zero, then $ P(D)=0 $. \qed
\end{corollary}

\begin{lemma}\label{lem:minors_G_23a}
	Every alternating dimap of genus zero that contains a proper $ \omega $-semiloop has $ G^{a}_{2,3} $ as a minor.
\end{lemma}

\begin{proof}
	From Table~\ref{tab:trial_edge_type}, we can see that a proper $ \omega $-semiloop can be obtained from a proper $ 1 $-semiloop $ e $ by applying the trial operation on $ e $ once. Likewise, we have $ (G_{1,3})^{\omega}=G^{a}_{2,3} $. Therefore, by triality and Lemma~\ref{lem:minors_G_13}, we complete the proof.
\end{proof}

Since $ G^{a}_{2,3} $ contains a proper $ 1 $-loop and a proper $ \omega^{2} $-loop, the following corollary follows from Definition~\ref{def:Extended_Tutte_Invariant}.
\begin{corollary}\label{cor:omega-semiloop}
	If $ x=0 $ or $ z=0 $, and there exists a proper $ \omega $-semiloop in an alternating dimap $ D $ of genus zero, then $ P(D)=0 $. \qed
\end{corollary}

\begin{lemma}\label{lem:minors_G_23c}
	Every alternating dimap of genus zero that contains a proper $ \omega^{2} $-semiloop has $ G^{c}_{2,3} $ as a minor.
\end{lemma}

\begin{proof}
	By triality and Lemma~\ref{lem:minors_G_23a} (or Lemma~\ref{lem:minors_G_13}).
\end{proof}

Since $ G^{c}_{2,3} $ contains a proper $ 1 $-loop and a proper $ \omega $-loop, the following corollary follows from Definition~\ref{def:Extended_Tutte_Invariant}.
\begin{corollary}\label{cor:omega2-semiloop}
	If $ x=0 $ or $ y=0 $, and there exists a proper $ \omega^{2} $-semiloop in an alternating dimap $ D $ of genus zero, then $ P(D)=0 $. \qed
\end{corollary}

\begin{lemma}\label{lem:minors}
	Every alternating dimap of genus zero that contains a proper edge has $ G_{1,3} $, $ G^{a}_{2,3} $ and $ G^{c}_{2,3} $ as minors.
\end{lemma}

\begin{proof}
	\begin{figure}[ht]
		\centering
		\includegraphics[width=0.85\textwidth]{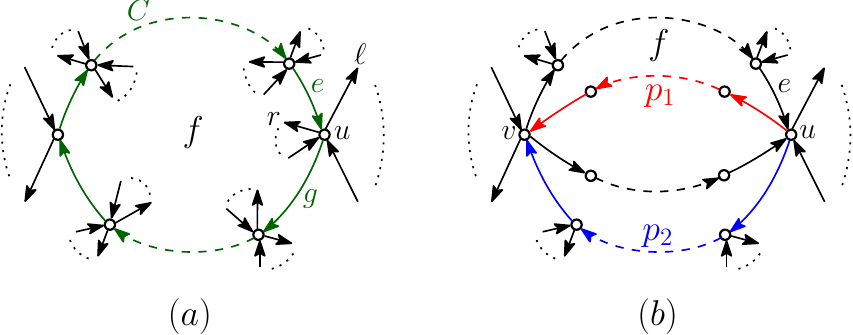}
		\caption{The clockwise face $ f $ in the proof of Lemma~\ref{lem:minors}}
		\label{fig:Minor_general_edge}
	\end{figure}

	Suppose $ D $ is an alternating dimap of genus zero. Let $ e \in E(D) $ be a proper edge that has  $ u \in V(D) $ as its head. Suppose $ f $ is a clockwise face (or an anticlockwise face with appropriate modifications) that contains $ e $, and $ C $ is the outer cycle of $ f $ in $ D $. We consider two cases as follows:
	\begin{enumerate}
		\item [i)] \textbf{There exists exactly one directed path between $ u $ and $ v $, for all $ u,v \in V(C) $} (see Figure~\ref{fig:Minor_general_edge}$ (a) $).
		\begin{itemize}
			\item Let $ \ell $ and $ r $ be the left successor and the right successor of $ e $, respectively. Suppose $ \ell \in E(C) $. Given that there is exactly one directed path between every pair of vertices in $ V(C) $, if $ \ell \in E(C) $, the edge $ \ell $ must be the next edge after $ e $ in $ C $. This implies that $ e $ is a proper $ \omega^{2} $-semiloop instead of a proper edge. Hence, $ \ell \notin E(C) $. Similar arguments show that $ r \notin E(C) $. The fact that $ \ell,r \notin E(C) $ implies that $ \hbox{deg}(u) \ge 6 $. By contracting every edge in $ E(C)\setminus e $ in $ D $, the edge $ e $ becomes a proper $ 1 $-semiloop. By Lemma~\ref{lem:minors_G_13}, we have $ G_{1,3} $ as a minor of $ D $.
			
			\item Let $ g \in E(C) $ be the edge directed out from $ u $. Suppose $ e $ and $ g $ partition $ T=I(u)\setminus e $ in $ D $ into two sets (based on the cyclic order of $ T $), $ (i) $ $ S_{c} $ that contains every edge directed into $ u $ that lies betwen $ g $ and $ e $ as we go from $ g $ to $ e $ in clockwise order around $ u $, and $ (ii) $ $ S_{a}=T\setminus S_{c} $. If we $ \omega $-reduce (respectively, $ \omega^{2} $-reduce) every edge in $ S_{c} $ (respectively, $ S_{a} $), the edge $ e $ is now a proper $ \omega $-semiloop (respectively, proper $ \omega^2 $-semiloop). By Lemma~\ref{lem:minors_G_23a} (respectively, Lemma~\ref{lem:minors_G_23c}), we have $ G^{a}_{2,3} $ (respectively, $ G^{c}_{2,3} $) as a minor of $ D $.
		\end{itemize}
		
		\item[ii)] \textbf{There exists more than one directed path between $ u $ and $ v $, for some $ u,v \in V(C) $} (see Figure~\ref{fig:Minor_general_edge}$ (b) $). Let $ p_{1} $ and $ p_{2} $ be two of the paths that are directed from $ u $ to $ v $, respectively.
		\begin{itemize}
			\item Contract every edge in $ p_{1} $, and all but one edge in $ p_{2} $; the remaining edge in $ p_{2} $ is a proper $ 1 $-semiloop. By Lemma~\ref{lem:minors_G_13}, we have $ G_{1,3} $ as a minor of $ D $.
			
			\item Recall that $ e \in E(\partial f) $. Suppose $ v \in V(\partial f) $ and $ E(p_{1}) \subset E(\partial f) $. Let $ h \in E(p_{1}) $ and $ h $ has $ v $ as its head. By Lemma~\ref{lem:bigger_c-face_to_alt_c}, a face $ f' $ of size exactly two can be obtained from $ f $, by a sequence of contraction operations. So, let $ V(\partial f')=\{u,v\} $ and $ E(\partial f')=\{e,h\} $. By $ \omega^{2} $-reducing every edge in $ I(u)\setminus e $, we have $ \hbox{deg}(u)=2 $, and $ h $ is now a proper $ \omega $-semiloop. By Lemma~\ref{lem:minors_G_23a}, we have $ G^{a}_{2,3} $ as a minor of $ D $.
			
			\item Let $ g $ be an anticlockwise face in $ D $, the proper edge $ e \in E(\partial g) $ and $ E(p_{2}) \subset E(\partial g) $. Suppose $ h \in E(p_{2}) $ and $ h $ has $ v $ as its head. By Lemma~\ref{lem:bigger_a-face_to_alt_a}, an anticlockwise face $ g' $ of size exactly two can be obtained from $ g $. So, let $ V(\partial g')=\{u,v\} $ and $ E(\partial g')=\{e,h\} $. By $ \omega $-reducing every edge in $ I(u)\setminus e $, we have $ \hbox{deg}(u)=2 $, and $ e $ is now a proper $ \omega^{2} $-semiloop. By Lemma~\ref{lem:minors_G_23c}, we have $ G^{c}_{2,3} $ as a minor of $ D $.
		\end{itemize}
	\end{enumerate}
\end{proof}

\begin{corollary}\label{cor:proper_edge}
	If $ x=0 $, $ y=0 $ or $ z=0 $, and there exists a proper edge in an alternating dimap $ D $ of genus zero, then $ P(D)=0 $.
\end{corollary}
 
\begin{proof}
	By Lemma~\ref{lem:minors}, if there exists a proper edge in an alternating dimap $ D $ of genus zero, then $ D $ contains $ G_{1,3} $, $ G^{a}_{2,3} $ and $ G^{c}_{2,3} $ as minors. Since a proper $ 1 $-loop, a proper $ \omega $-loop and a proper $ \omega^{2} $-loop may each be found in at least one of these minors, by Definition~\ref{def:Extended_Tutte_Invariant}, we have $ P(D)=0 $ when at least one of $ x,y,z $ is zero.
\end{proof}

By Lemmas~\ref{lem:minors_G_13}, \ref{lem:minors_G_23a}, \ref{lem:minors_G_23c} and \ref{lem:minors}, we obtain the following corollary.
\begin{corollary}\label{cor:minors}
	Every alternating dimap of genus zero that contains a non-triloop edge has $ G_{1,3} $ or $ G^{a}_{2,3} $ or $ G^{c}_{2,3} $ as a minor. \qed
\end{corollary}

Further from Theorem~\ref{thm:Extended_Tutte_Invariant}, we now discuss the the properties of alternating dimaps of genus zero that are required, in order to obtain a non-trivial invariant, when at least one of $ x,y,z $ is zero.
\begin{theorem}\label{thm:x=y=z=0}
	Let $ S=(w,x,y,z,a,b,c,d,e,f,g,h,i,j,k,l) $ be a parameter sequence such that $ w \ne 0 $ and $ x=y=z=0 $. A function $ P $ is an extended Tutte invariant with respect to $ S $ for every alternating dimap $ D $ of genus zero if and only if
	\begin{equation*}
		P(D) = \left\{ 
		\begin{array}{ll}
			w^{k(D)},	& \text{if \instar(D)=\aface(D)=\cface(D)=k(D),}\\
			0,			& \text{otherwise,}
		\end{array} \right.
	\end{equation*}
	with $ a=f=h=0 $.
\end{theorem}

\begin{proof}
	Let $ S $ be as stated and $ D $ be an alternating dimap of genus zero. We first prove the forward implication. Note that all the derived polynomials must be equal for $ P(D) $ to be an extended Tutte invariant. By using Lemma~\ref{lem:derived_polynomials_G_{1,3}_G^{a}_{2,3}_G^{c}_{2,3}} and the fact that $ w \ne 0 $ and $ x=y=z=0 $, we obtain $ a=f=h=0 $ as desired.
	
	If $ \instar(D)=\aface(D)=\cface(D)=k(D) $, then $ D $ is a disjoint union of ultraloops. By Definition~\ref{def:Extended_Tutte_Invariant}, we have $ P(D)=w^{k(D)} $.
	
	We next show $ P(D)=0 $ in the following cases.
	\begin{enumerate}
		\item[i)] $ \instar(D) \ne k(D) = \aface(D) = \cface(D) $. There is a proper $ 1 $-loop in $ D $. By (\hyperref[def:Extended_Tutte_Invariant]{ETI2}) and using $ x=0 $, we have $ P(D)=0 $.
		
		\item[ii)] $ \aface(D) \ne k(D) = \instar(D) = \cface(D) $. A similar approach to (i) gives $ P(D)=0 $.
		
		\item[iii)] $ \cface(D) \ne k(D) = \instar(D) = \aface(D) $. A similar approach to (i) gives $ P(D)=0 $.
		
		\item[iv)] $ \aface(D) \ne k(D) $, $ \cface(D) \ne k(D) = \instar(D) $. If every edge in $ D $ is a triloop, the result is trivial. Otherwise, there exists a proper $ 1 $-semiloop in $ D $. By Corollary~\ref{cor:1-semiloop}, we have $ P(D)=0 $.
		
		\item[v)] $ \instar(D) \ne k(D) $, $ \cface(D) \ne k(D) = \aface(D) $. A similar approach to (iv) gives $ P(D)=0 $.
		
		\item[vi)] $ \instar(D) \ne k(D) $, $ \aface(D) \ne k(D) = \cface(D) $. A similar approach to (iv) gives $ P(D)=0 $.
				
		\item[vii)] $ \instar(D) \ne k(D) $, $ \aface(D) \ne k(D) $ and $ \cface(D) \ne k(D) $. If every edge in $ D $ is a triloop, the result is trivial. Otherwise, Corollaries~\ref{cor:1-semiloop}, \ref{cor:omega-semiloop}, \ref{cor:omega2-semiloop} or \ref{cor:proper_edge} gives $ P(D)=0 $.
	\end{enumerate}

	Conversely, we show that
	\begin{equation*}
		P(D) = \left\{ 
		\begin{array}{ll}
			w^{k(D)},	& \text{if \instar(D)=\aface(D)=\cface(D)=k(D),}\\
			0,			& \text{otherwise,}
		\end{array} \right.
	\end{equation*}
	with $ a=f=h=0 $ is an extended Tutte invariant with respect to $ S $ for every alternating dimap $ D $, as in Definition~\ref{def:Extended_Tutte_Invariant}. Based on the proof in Theorem~\ref{thm:Extended_Tutte_Invariant}, it is clear that $ P(D) $ is a multiplicative invariant.
		
	We now show that $ P(D) $ satisfies (\hyperref[def:Extended_Tutte_Invariant]{ETI1}) to (\hyperref[def:Extended_Tutte_Invariant]{ETI8}) by using induction on $ \abs{E(D)}=m $. When $ m=0 $, we have $ P(D)=1 $ and the result for $ m=0 $ follows. So, suppose $ m>0 $ and the result holds for every alternating dimap of genus zero that has size less than $ m $.
	
	Suppose $ \instar(D)=\aface(D)=\cface(D)=k(D) $. We have $ D $ as a disjoint union of ultraloops. Deletion of an ultraloop $ r $ from $ D $ reduces the number of components of $ D $ by 1. Thus,
	\begin{align*}
		P(D) & = w^{k(D)} = w \cdot w^{k(D)-1} = w \cdot w^{k(D\setminus r)} = w \cdot P(D\setminus r),
	\end{align*}
	where the last equality uses the inductive hypothesis.
	
	Note that for $ \mu \in \{1,\omega,\omega^{2}\} $, it is possible to have $ \instar(D[\mu]r)=\aface(D[\mu]r)=\cface(D[\mu]r)=k(D[\mu]r) $ by $ \mu $-reducing $ r $. Hence, we may have to consider more than one scenario for the remaining cases. Let $ r \in E(D) $.
	\begin{enumerate}
		\item[i)] $ r $ is a proper $ 1 $-loop.
		\begin{enumerate}
			\item $ \instar(D[1]r)=\aface(D[1]r)=\cface(D[1]r)=k(D[1]r) $.		
			\begin{align*}
				P(D) & = 0 = x \cdot w^{k(D[1]r)} = x \cdot P(D[1]r),
			\end{align*}
			where the penultimate equality uses $ x=0 $ and the last equality uses the inductive hypothesis.
			
			\item Otherwise,
			\begin{align*}
				P(D) & = 0 = x \cdot 0 = x \cdot P(D[1]r),
			\end{align*}
			where the last equality uses the inductive hypothesis.
		\end{enumerate}
	
		\item[ii)] $ r $ is a proper $ 1 $-semiloop. In $ D[1]r $, the number of components and in-stars are both increased by $ 1 $. In $ D[\omega]r $ and $ D[\omega^{2}]r $, the number of c-faces and a-faces are reduced by $ 1 $, respectively.
		\begin{enumerate}
			\item $ \instar(D[1]r)=\aface(D[1]r)=\cface(D[1]r)=k(D[1]r) $.		
			\begin{align*}
				P(D) & = 0 = a \cdot w^{k(D[1]r)} + b \cdot 0 + c \cdot 0 \tag{since $ a=0 $}\\
					& = a \cdot P(D[1]r) + b \cdot P(D[\omega]r) + c \cdot P(D[\omega^{2}]r) \tag[.]{by the inductive hypothesis}
			\end{align*}
			
			\item Otherwise,
			\begin{align*}
				P(D) & = 0 = a \cdot 0 + b \cdot 0 + c \cdot 0\\
					& = a \cdot P(D[1]r) + b \cdot P(D[\omega]r) + c \cdot P(D[\omega^{2}]r) \tag[.]{by the inductive hypothesis}
			\end{align*}
		\end{enumerate}

		\item[iii)] $ r $ is a proper edge.
		\begin{align*}
			P(D) & = 0 = j \cdot 0 + k \cdot 0 + l \cdot 0\\
				& = j \cdot P(D[1]r) + k \cdot P(D[\omega]r) + l \cdot P(D[\omega^{2}]r) \tag[.]{by the inductive hypothesis}
		\end{align*}
	\end{enumerate}

	The arguments for the other four cases are similar. See~\cite{YowPHD2019} for details.
\end{proof}

Next, we have the following three results in which two of $ x,y,z $ are zero.
\begin{theorem}\label{thm:ETI_x=y=0}
	Let $ S=(w,x,y,z,a,b,c,d,e,f,g,h,i,j,k,l) $ be a parameter sequence such that $ w,z \ne 0 $ and $ x=y=0 $. A function $ P $ is an extended Tutte invariant with respect to $ S $ for every alternating dimap $ D $ of genus zero if and only if
	\begin{equation*}
		P(D) = \left\{ 
		\begin{array}{ll}
			w^{k(D)} \cdot z^{\cface(D)-k(D)},	& \text{if \instar(D)=\aface(D)=k(D),}\\
			0,									& \text{otherwise,}
		\end{array} \right.
	\end{equation*}
	with $ a=c=d=f=h=0 $.
\end{theorem}

\begin{proof}
	The proof is similar to the proof in Theorem~\ref{thm:x=y=z=0}, with some extra routine details.
\end{proof}

Triality leads to the following corollaries, for $ x=z=0 $ and $ y=z=0 $, respectively.
\begin{corollary}
	Let $ S=(w,x,y,z,a,b,c,d,e,f,g,h,i,j,k,l) $ be a parameter sequence.
	\begin{enumerate}
		\item[a)] For $ w,y \ne 0 $ and $ x=z=0 $, a function $ P $ is an extended Tutte invariant with respect to $ S $ for every alternating dimap $ D $ of genus zero if and only if
		\begin{equation*}
		P(D) = \left\{ 
			\begin{array}{ll}
				w^{k(D)} \cdot y^{\aface(D)-k(D)},	& \text{if \instar(D)=\cface(D)=k(D),}\\
				0,									& \text{otherwise,}
			\end{array} \right.
		\end{equation*}
		with $ a=b=f=g=h=0 $.
	
	\item[b)] For $ w,x \ne 0 $ and $ y=z=0 $, a function $ P $ is an extended Tutte invariant with respect to $ S $ for every alternating dimap $ D $ of genus zero if and only if
		\begin{equation*}
		P(D) = \left\{ 
			\begin{array}{ll}
				w^{k(D)} \cdot x^{\instar(D)-k(D)},	& \text{if \aface(D)=\cface(D)=k(D),}\\
				0,									& \text{otherwise,}
			\end{array} \right.
		\end{equation*}
		with $ a=e=f=h=i=0 $.\qed
	\end{enumerate}
\end{corollary}

Lastly, we investigate cases where exactly one of the three variables is zero.
\begin{theorem}\label{thm:ETI_x=0}
	Let $ S=(w,x,y,z,a,b,c,d,e,f,g,h,i,j,k,l) $ be a parameter sequence such that $ w,y,z \ne 0 $ and $ x=0 $. A function $ P $ is an extended Tutte invariant with respect to $ S $ for every alternating dimap $ D $ of genus zero if and only if
	\begin{equation*}
		P(D) = \left\{ 
		\begin{array}{ll}
			w^{k(D)} \cdot y^{\aface(D)-k(D)} \cdot z^{\cface(D)-k(D)},	& \text{if \instar(D)=k(D),}\\
			0,															& \text{otherwise,}
		\end{array} \right.		
	\end{equation*}
	with $ d=f=g=h=j=0 $ and $ yz=aw+by+cz $.
\end{theorem}

\begin{proof}
	The proof is similar to the proof in Theorem~\ref{thm:x=y=z=0}, with some extra routine details~\cite{YowPHD2019}.
\end{proof}

Similarly, by using triality, we have
\begin{corollary}\label{cor:ETI_y=0_z=0}
	Let $ S=(w,x,y,z,a,b,c,d,e,f,g,h,i,j,k,l) $ be a parameter sequence.
	\begin{enumerate}
		\item[a)] For $ w,x,z \ne 0 $ and $ y=0 $, a function $ P $ is an extended Tutte invariant with respect to $ S $ for every alternating dimap $ D $ of genus zero if and only if
		\begin{equation*}
			P(D) = \left\{ 
			\begin{array}{ll}
				w^{k(D)} \cdot x^{\instar(D)-k(D)} \cdot z^{\cface(D)-k(D)},	& \text{if \aface(D)=k(D),}\\
				0,																& \text{otherwise,}
			\end{array} \right.		
		\end{equation*}
		with $ a=c=h=i=l=0 $ and $ xz=dz+ex+fw $.
		
		\item[b)] For $ w,x,y \ne 0 $ and $ z=0 $, a function $ P $ is an extended Tutte invariant with respect to $ S $ for every alternating dimap $ D $ of genus zero if and only if
		\begin{equation*}
			P(D) = \left\{ 
			\begin{array}{ll}
				w^{k(D)} \cdot x^{\instar(D)-k(D)} \cdot y^{\aface(D)-k(D)},	& \text{if \cface(D)=k(D),}\\
				0,																& \text{otherwise.}
			\end{array} \right.		
		\end{equation*}
		with $ a=b=e=f=k=0 $ and $ xy=gy+hw+ix $.\qed
	\end{enumerate}
\end{corollary}

\subsection{Restricted Alternating Dimaps, Independent Parameters}\label{sec:well_defined_ETI}

In order for an extended Tutte invariant to exist for an alternating dimap $ D $, every edge-ordering of $ D $ must give an identical derived polynomial, when $ D $ is reduced using these edge-orderings.

In \cref{sec:characterisations_of_ETI}, we identified restrictions on the parameters that ensure extended Tutte invariants exist for all alternating dimaps when the restrictions are satisfied (see Theorem~\ref{thm:Extended_Tutte_Invariant}). We now investigate the conditions on an alternating dimap that are required in order to obtain an extended Tutte invariant for it, without any restriction on the parameters. The fact that no restriction is imposed on the parameters implies that the variables $ w,x,y,z,a,b,c,d,e,f,g,h,i,j,k,l $ are all independent, and will be treated as indeterminates. 

To formalise this distinction, we need a more specific extended Tutte invariant. The \emph{complete} extended Tutte invariant of an alternating dimap takes values in a ring $ \mathbb{E}[w,x,y,z,a,b,c,d,e,f,g,h,i,j,k,l] $, where $ \mathbb{E} $ is a field. The ring is considered to be a subset of the field of fractions $ \mathbb{F} \coloneqq \mathbb{E}(w,x,y,z,a,b,c,d,e,f,g,h,i,j,k,l) $ whose numerators and denominators are in $ \mathbb{E}[w,x,y,z,a,b,c,d,e,f,g,h,i,j,k,l] $.

We will determine the domain of the complete extended Tutte invariant, which is the set of alternating dimaps for which it exists.

Let $ D $ be an alternating dimap. For $ i \in \{1,2, \ldots, \abs{E(D)} \} $ and $ \mu_{i} \in \{1, \omega, \omega^{2}\} $, a \emph{reduction sequence} for a given edge-ordering $ \mathcal{O}=e_{1}e_{2} \ldots e_{i} $ of $ D $ is a sequence of reductions $ \mathcal{R}=\mu_{1},\mu_{2}, \ldots ,\mu_{i} $. By reducing $ D $ using the edge-ordering $ \mathcal{O} $ and the reduction sequence $ \mathcal{R} $, we obtain the minor $ D[\mu_{1}]e_{1}[\mu_{2}]e_{2} \ldots [\mu_{i}]e_{i} $, which is denoted by $ D[\mathcal{R}]\mathcal{O} $.

For a given edge-ordering, an extended Tutte invariant is constructed using a set of sequences of reductions of the edges, where the edges are reduced in the given order. For each sequence of reductions, a factor is introduced each time an edge is reduced in the sequence. For instance, the factor of $ x $ is introduced when a proper $ 1 $-loop is reduced in Definition~\ref{def:Extended_Tutte_Invariant}. Note that if more than one reduction is performed on an edge (i.e., a non-triloop edge in Definition~\ref{def:Extended_Tutte_Invariant}), the type of reduction operation determines the factor that will be introduced for each minor.

For $ i \in \{1,2,\ldots,\abs{E(D)}\} $, suppose $ H $ is a minor of $ D $ that is obtained by reducing the first $ i $ edges of $ D $. Then, the first $ i $ factors introduced by these $ i $ reductions form the \emph{monomial of} $ H $ with respect to $ D $ and the reductions used. In Figure~\ref{fig:monomial}, the first two edges of an alternating dimap $ D $ are reduced in the extended Tutte invariant by using $ \mathcal{O}=pq $.
\begin{figure}[t]
	\centering
	\includegraphics[scale=0.75]{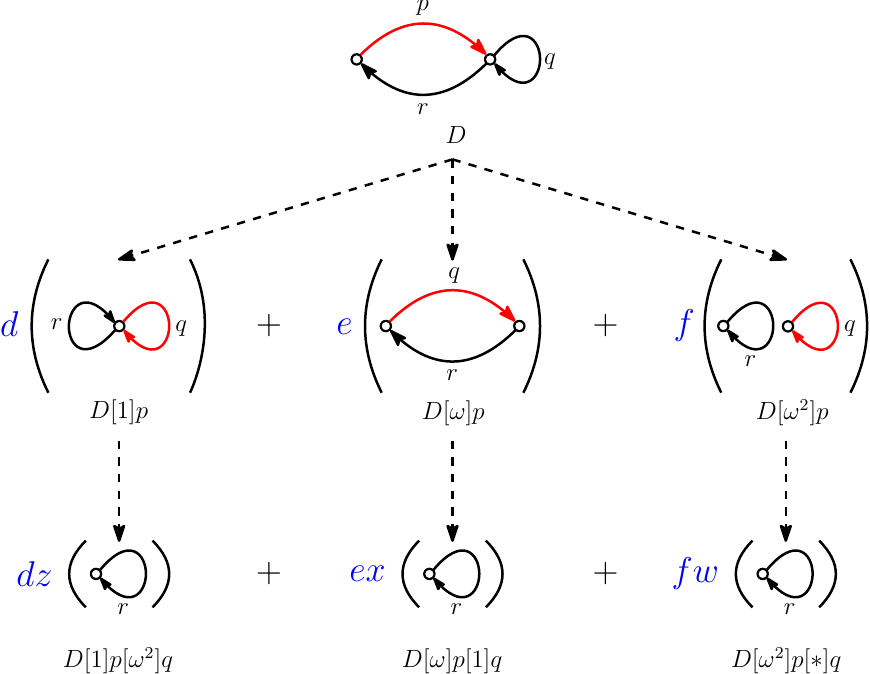}
	\caption{Reductions on the first two edges of an alternating dimap $ D $}
	\label{fig:monomial}
\end{figure}
Since only one edge $ p $ is reduced to obtain the minor $ D[\omega]p $, the factor $ e $ is also the monomial of this minor. On the other hand, for the minor $ D[\omega^{2}]p[*]q $, two factors $ f $ and $ w $ are obtained. Hence, we have $ fw $ as the monomial of $ D[\omega^{2}]p[*]q $.

\begin{proposition}\label{pro:well_defined}
	If an extended Tutte invariant is well defined for an alternating dimap $ D $, then this holds for any minor of $ D $. \qed
\end{proposition}

\begin{lemma}\label{lem:not_well_defined_G{1,3}}
	The complete extended Tutte invariant is not well defined for the alternating dimap $ G_{1,3} $.
\end{lemma}

\begin{proof}[Proof (sketch)]
	By reducing the alternating dimap $ G_{1,3} $ using different edge-orderings, we obtain two distinct derived polynomials $ wyz $ and $ aww + bwy + cwz $.
\end{proof}

\begin{lemma}\label{lem:not_well_defined_G{2,3}}
	The complete extended Tutte invariant is not well defined for either of the alternating dimaps $ G_{2,3} $.
\end{lemma}

\begin{proof}[Proof (sketch)]
	Since there exist two possibilities for $ G_{2,3} $, we consider two alternating dimaps $ G^{a}_{2,3} $ and $ G^{c}_{2,3} $ separately.
	
	By reducing the alternating dimap $ G^{a}_{2,3} $ using different edge-orderings, we obtain two distinct derived polynomials $ wxz $ and $ dwz + ewx + fww $. Similarly, the alternating dimap $ G^{c}_{2,3} $ gives $ wxy $ and $ gwy + hww + iwx $.
\end{proof}

We now give the conditions required in order to obtain a well defined complete extended Tutte invariant.
\begin{theorem}\label{thm:triloops_iff_ETI}
	The complete extended Tutte invariant is well defined for an alternating dimap $ D $ if and only if $ D $ contains only triloops.
\end{theorem}

\begin{proof}
	The forward implication is proved by contrapositive. Let $ D $ be an alternating dimap that contains at least one non-triloop edge. By Corollary~\ref{cor:minors}, the alternating dimap $ D $ contains $ G_{1,3} $ or $ G_{2,3} $ as a minor. By Lemma~\ref{lem:not_well_defined_G{1,3}} and Lemma~\ref{lem:not_well_defined_G{2,3}}, the complete extended Tutte invariant is not well defined for $ G_{1,3} $ and $ G_{2,3} $, respectively. By Proposition~\ref{pro:well_defined}, the complete extended Tutte invariant is then not well defined for $ D $. Hence, the forward implication follows.
	
	Conversely, suppose an alternating dimap $ D $ contains only triloops. This implies that every edge in each connected component of $ D $ is of the same type. By Definition~\ref{def:Extended_Tutte_Invariant}, each time a triloop is chosen and reduced, one factor $ w,\ x,\ y $ or $ z $ is introduced. Suppose $ r \in E(D) $ is the first edge in a given edge-ordering. By using the given edge-ordering, the triloop $ r $ is deleted after the first reduction operation. All the other edges in $ D\setminus r $ remain as triloops of the same type as they were in $ D $. This is always true regardless of which edge is first reduced in $ D $. In other words, the edge-ordering is inconsequential. In addition, the final edge to be reduced in each component is always an ultraloop (an improper triloop). Since the complete extended Tutte invariant for alternating dimaps is multiplicative, and multiplication is commutative, the complete extended Tutte invariant is well defined for $ D $.
\end{proof}

\section{Tutte Invariants That Extend the Tutte Polynomial}\label{sec:Tutte_invariant_vs_Tutte_polynomial}

Recall from \cite{Tutte1954} that the Tutte polynomial $ T(G;x,y) $ of a graph $ G $ has the following deletion-contraction recurrence, for any $ e \in E(G) $: 
	\[ 
	T(G;x,y) = \left\{ 
	\begin{array}{ll}
		1,									& \text{if $ G $ is empty,}\\
		x\cdot T(G/e;x,y),					& \text{if $ e $ is a coloop,}\\
		y\cdot T(G \setminus e;x,y),		& \text{if $ e $ is a loop,}\\
		T(G \setminus e;x,y)+T(G/e;x,y),	& \text{otherwise.}
	\end{array} \right.
	\]

In addition to the extended Tutte invariant in Definition~\ref{def:Extended_Tutte_Invariant}, Farr~\cite{Farr2018} defined two other Tutte invariants, namely $ T_c(D;x,y) $ and $ T_a(D;x,y) $ which are analogues of the Tutte polynomial. Note that $ T_c(D;x,y) $ and $ T_a(D;x,y) $ are two special cases of extended Tutte invariants.

In this section, we discuss these two invariants for alternating dimaps that are 2-cell embedded on an orientable surface of genus zero.
\begin{definition}\label{def:T_c}
	A \emph{c-Tutte invariant} for a class $ \mathcal{A} $ of alternating dimaps is a multiplicative invariant $ T_{c} $ for $ \mathcal{A} $ such that, for any alternating dimap $ D \in \mathcal{A} $ and $ e \in E(D) $,
	\begin{enumerate}
		\item if $ e $ is an ultraloop,
		\begin{equation*}
			T_c(D;x,y)=T_{c}(D\setminus e;x,y), \tag{TC1}
		\end{equation*}
		
		\item if $ e $ is a proper 1-loop or a proper $ \omega $-semiloop,
		\begin{equation*}
			T_c(D;x,y)=x\cdot T_{c}(D[\omega^{2}]e;x,y), \tag{TC2}
		\end{equation*}
		
		\item if $ e $ is a proper $ \omega $-loop or a proper 1-semiloop,
		\begin{equation*}
			T_c(D;x,y)=y\cdot T_{c}(D[1]e;x,y), \tag{TC3}
		\end{equation*}
		
		\item if $ e $ is a proper $ \omega^{2} $-loop or a proper $ \omega^{2} $-semiloop,
		\begin{equation*}
			T_c(D;x,y)=T_{c}(D[\omega]e;x,y), \tag{TC4}
		\end{equation*}
		
		\item otherwise,
		\begin{equation*}
			T_c(D;x,y)=T_{c}(D[1]e;x,y)+T_{c}(D[\omega^{2}]e;x,y). \tag{TC5}
		\end{equation*}
	\end{enumerate}
\end{definition}

\begin{definition}\label{def:T_a}
	An \emph{a-Tutte invariant} for a class $ \mathcal{A} $ of alternating dimaps is a multiplicative invariant $ T_{a} $ for $ \mathcal{A} $ such that, for any alternating dimap $ D \in \mathcal{A} $ and $ e \in E(D) $,
	\begin{enumerate}
		\item if $ e $ is an ultraloop,
		\begin{equation*}
			T_a(D;x,y)=T_{a}(D\setminus e;x,y),
		\end{equation*}
		
		\item if $ e $ is a proper 1-loop or a proper $ \omega^{2} $-semiloop,
		\begin{equation*}
			T_a(D;x,y)=x\cdot T_{a}(D[\omega]e;x,y),
		\end{equation*}
		
		\item if $ e $ is a proper $ \omega^{2} $-loop or a proper 1-semiloop,
		\begin{equation*}
			T_a(D;x,y)=y\cdot T_{a}(D[1]e;x,y),
		\end{equation*}
		
		\item if $ e $ is a proper $ \omega $-loop or a proper $ \omega $-semiloop,
		\begin{equation*}
			T_a(D;x,y)=T_{a}(D[\omega^{2}]e;x,y),
		\end{equation*}
		
		\item otherwise,
		\begin{equation*}
		T_a(D;x,y)=T_{a}(D[1]e;x,y)+T_{a}(D[\omega]e;x,y).
		\end{equation*}
	\end{enumerate}
\end{definition}

\noindent
\textbf{Remark:} For the reduction of a triloop $ e \in E(D) $, we have $ D[*]e=D[1]e=D[\omega]e=D[\omega^{2}]e=D\backslash e $.\\

\begin{theorem}\emph{\cite[Theorem 5.2]{Farr2018}}\label{thm:T_G=T_c_alt_c(G)}
	For any plane graph $ G $,
	\begin{equation*}
		T(G;x,y)=T_c(\emph{\hbox{alt}}_c(G);x,y)=T_a(\emph{\hbox{alt}}_a(G);x,y).
	\end{equation*}
\end{theorem}

\subsection{Tutte Invariants with Dependent Paramaters for Arbitrary Alternating Dimaps}\label{subsec:c_a_Tutte_all}

We now determine when a c-Tutte invariant is well defined for all alternating dimaps of genus zero, using our results on extended Tutte invariants.

\begin{proposition}\label{pro:T_c_general}
	The c-Tutte invariant is well defined for all alternating dimaps of genus zero if and only if 
	\begin{equation*}
		x=\frac{1\pm \sqrt{3}i}{2}, \ y=\frac{1\mp \sqrt{3}i}{2}.
	\end{equation*}	
\end{proposition}

\begin{proof}[Proof]
	We first consider the definitions of the extended Tutte invariant (see Definition~\ref{def:Extended_Tutte_Invariant}) and the c-Tutte invariant (see Definition~\ref{def:T_c}). Since the variables $ x $ and $ y $ are used in both definitions, we use $ \alpha $ and $ \beta $ instead of the variables $ x $ and $ y $, respectively, that are used for the c-Tutte invariant. By comparing the recurrences in the two definitions, we see that the c-Tutte invariant is an extended Tutte invariant with parameters
	\begin{equation}\label{eqn:CG1}
		x=f=\alpha, \ y=a=\beta, \ w=z=h=j=l=1, \ b=c=d=e=g=i=k=0.
	\end{equation}
	
	\begin{itemize}
		\item[i)] Suppose $ \alpha,\beta \ne 0 $. In this case, the hypothesis of Theorem~\ref{thm:Extended_Tutte_Invariant} is satisfied (since $ w=z=1 \ne 0 $, $ x=\alpha \ne 0 $ and $ y=\beta \ne 0 $). By substituting the respective values in (\ref{eqn:CG1}) into the necessary conditions in (\ref{eqn:necessary_conditions_for_ETI_1})--(\ref{eqn:necessary_conditions_for_ETI_4}), and solving the equations, we obtain
		\begin{equation}\label{eqn:CG2}
		\alpha\beta=\alpha+\beta=1.
		\end{equation}
		By solving (\ref{eqn:CG2}), and using the fact that $ \alpha=x $ and $ \beta=y $ in the c-Tutte invariant, we have
		\begin{equation*}\label{eqn:CG3}
			x=\frac{1\pm \sqrt{3}i}{2}, \ y=\frac{1\mp \sqrt{3}i}{2}.
		\end{equation*}
		
		\item[ii)] Suppose $ \alpha = 0 $ and $ \beta \ne 0 $. In this case, the hypothesis of Theorem~\ref{thm:ETI_x=0} is satisfied (since $ w=z=1 \ne 0 $, $ x=\alpha = 0 $ and $ y=\beta \ne 0 $). However, by Theorem~\ref{thm:ETI_x=0}, the fact that $ h=j=1 \ne 0 $ implies that we do not get a well defined extended Tutte invariant.
		
		\item[iii)] Suppose $ \alpha \ne 0 $ and $ \beta = 0 $. In this case, the hypothesis of Corollary~\ref{cor:ETI_y=0_z=0}(a) is satisfied (since $ w=z=1 \ne 0 $, $ x=\alpha \ne 0 $ and $ y=\beta = 0 $). However, by Corollary~\ref{cor:ETI_y=0_z=0}(a), the fact that $ h=l=1 \ne 0 $ implies that we do not get a well defined extended Tutte invariant.
		
		\item[iv)] Suppose $ \alpha=\beta = 0 $. In this case, the hypothesis of Theorem~\ref{thm:ETI_x=y=0} is satisfied (since $ w=z=1 \ne 0 $, $ x=\alpha = 0 $ and $ y=\beta = 0 $). However, by Theorem~\ref{thm:ETI_x=y=0}, the fact that $ h=1 \ne 0 $ implies that we do not get a well defined extended Tutte invariant.
	\end{itemize}
	
	The backward implication follows, by Theorem~\ref{thm:Extended_Tutte_Invariant}.
\end{proof}

\begin{corollary}\label{cor:T_c_general}
	The only c-Tutte invariants that are well defined for all alternating dimaps $ D $ of genus zero are
	\begin{equation*}
		T_{c}(D;\frac{1\pm \sqrt{3}i}{2},\frac{1\mp \sqrt{3}i}{2}) =
		\left( \frac{1\pm \sqrt{3}i}{2} \right)^{\instar(D)-\aface(D)}.
	\end{equation*}
	\qed
\end{corollary}

Note that \[ x=\frac{1+\sqrt{3}i}{2}, \ y=\frac{1-\sqrt{3}i}{2} \] are the two primitive sixth roots of unity. These two points satisfy the equation $ (x-1)(y-1)=1 $, so they lie on the hyperbola $ H_{1}:=\{(x,y):(x-1)(y-1)=1\} $, on which $ T(G;x,y) $ and hence $ T_c(\hbox{alt}_c(G);x,y) $ are easy to evaluate~\cite{Welsh1993}.

By using a similar approach, we have the following proposition for the a-Tutte invariant.
\begin{proposition}
	The a-Tutte invariant is well defined for all alternating dimaps of genus zero if and only if 
	\begin{equation*}
		x=\frac{1\pm \sqrt{3}i}{2}, \ z=\frac{1\mp \sqrt{3}i}{2}.
	\end{equation*}	
\end{proposition}
 
\begin{proof}
	Use similar arguments as in the proof of Proposition~\ref{pro:T_c_general}.
\end{proof}

\begin{corollary}\label{cor:T_a_general}
	The only a-Tutte invariants that are well defined for all alternating dimaps $ D $ of genus zero are
	\begin{equation*}\label{eqn:Extended_Tutte_Invariant_Ga2}
		T_{a}(D;\frac{1\pm \sqrt{3}i}{2},\frac{1\mp \sqrt{3}i}{2}) = 
		\left( \frac{1\pm \sqrt{3}i}{2} \right)^{\instar(D)-\cface(D)}.
	\end{equation*}
	\qed
\end{corollary}

\subsection{c-Tutte Invariants for Restricted Alternating Dimaps}

The c-Tutte invariant and the a-Tutte invariant are closely related. Once a problem is solved for one of these invariants, it can then be solved for the other by some appropriate modifications, as evidenced in \cref{subsec:c_a_Tutte_all}. Hence, we only focus on the c-Tutte invariant from now onwards.

A \emph{c-cycle block} (respectively, an \emph{a-cycle block}) of an alternating dimap $ D $ is a block that is a clockwise face (respectively, an anticlockwise face) of $ D $ that has the same number of vertices as edges. Such a block is a directed cycle of $ D $.

A \emph{c-simple alternating dimap} (see Figure~\ref{fig:c-simple alternating dimap})
\begin{figure}[t]
	\centering
	\includegraphics[width=0.75\textwidth]{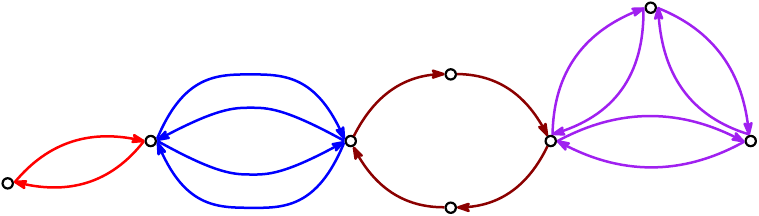}
	\caption{A c-simple alternating dimap with four blocks}
	\label{fig:c-simple alternating dimap}
\end{figure}
is a loopless alternating dimap of genus zero in which every block is either:
\begin{itemize}
	\item[i)] a c-cycle block, or
	\item[ii)] an element of $ \hbox{alt}_c(\mathcal{G}) $,
\end{itemize}
and there exists no block within a clockwise face of any other block.

A \emph{c-alternating dimap} is an alternating dimap of genus zero that can be obtained from a c-simple alternating dimap by adding some c-multiloops within some anticlockwise faces of the c-simple alternating dimap. Hence, a c-simple alternating dimap is merely a c-alternating dimap without any loops.

Let $ A $ denote the set of cutvertices and $ B $ denote the set of blocks of a c-alternating dimap $ H $. We construct the \emph{c-block graph} of $ H $ with vertex set $ A \cup B $ as follows: $ a_{i} \in A $ and $ b_{j} \in B $ are adjacent if block $ b_{j} $ of $ H $ contains the cutvertex $ a_{i} $ of $ H $. The construction of the c-block graph of a c-alternating dimap is the same as the construction of the \emph{block graph} of a graph. Hence, the c-block graph of a connected c-alternating dimap is a tree. An example of a c-alternating dimap and the corresponding c-block graph is shown in Figure~\ref{fig:c-alternating_dimap_and_c-block}.
\begin{figure}[t]
	\centering
	\includegraphics[width=0.75\textwidth]{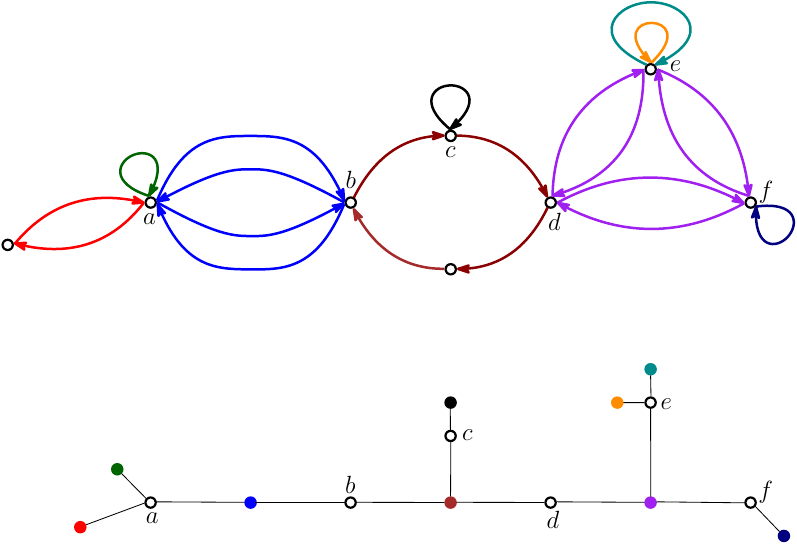}
	\caption{A c-alternating dimap and its c-block graph}
	\label{fig:c-alternating_dimap_and_c-block}
\end{figure}

\begin{lemma}\label{lem:alt_c_size_two}
	Let $ D $ be an alternating dimap. Every clockwise face of $ D $ has size exactly two if and only if there exists an undirected orientably embedded graph $ G $ such that $ D \cong \hbox{alt}_c(G) $.
\end{lemma}

\begin{proof}
	We first prove the forward implication. Given an alternating dimap $ D $ where all of its clockwise faces have size exactly two, we construct an undirected graph $ G $ as follows. Let $ V(G)=V(D) $. For each clockwise face of $ D $, a new edge of $ G $ that is incident with the same endvertices (i.e., the two vertices incident with this face, which may coincide) is added in $ G $ such that the new edge is within the clockwise face of $ D $. Each clockwise face contains exactly one edge in this way, therefore edges in $ G $ do not intersect. Hence, we obtain an undirected embedded graph $ G $ such that $ D \cong \hbox{alt}_c(G) $.
	
	Conversely, if there exists an undirected graph $ G $ such that an alternating dimap $ D \cong \hbox{alt}_c(G) $, every clockwise face of $ D $ has size exactly two, by the definition of $ \hbox{alt}_c(G) $.
\end{proof}

Two graphs $ G $ and $ H $ are \emph{codichromatic} (or \emph{Tutte equivalent}) if $ T(G;x,y)=T(H;x,y) $. Tutte~\cite{Tutte1954,Tutte1967,Tutte1974} proved that the Tutte polynomial is multiplicative over blocks.
\begin{theorem*}[\textbf{Tutte 1954}]
	Let $ G $ be the union of two subgraphs $ H $ and $ K $ having no common edge and at most one common vertex. Then,
	\begin{equation*}
		T(G;x,y)=T(H;x,y) \cdot T(K;x,y).
	\end{equation*}
\end{theorem*}

\begin{proposition}\label{pro:c-union_c-alt_dimap}
	The c-union of two c-alternating dimaps is also a c-alternating dimap. \qed
\end{proposition}

\begin{corollary}\label{cor:c-union_c-simple_alt_dimap}
	The c-union of two c-simple alternating dimaps is also a c-simple alternating dimap. \qed
\end{corollary}

\begin{proposition}
	Let $ D_{1} $ and $ D_{2} $ be alternating dimaps and $ D = D_{1} \cup_{c} D_{2} $. Then the set of blocks in $ D $ is the union of the sets of blocks of $ D_{1} $ and $ D_{2} $. \qed
\end{proposition}

In the following lemmas, we show that the c-Tutte invariant is well defined for any alternating dimap (of genus zero) that is either a c-cycle block or a c-multiloop.
\begin{lemma}\label{lem:c-cycle_blocks}
	Let $ D $ be an alternating dimap that is a c-cycle block of size $ m \ge 1 $. Then, \[ T_{c}(D;x,y)=x^{m-1}. \]
\end{lemma}

\begin{proof}
	Induction on $ m $.
\end{proof}

\begin{corollary}\label{cor:c-cycle_blocks}
	Let $ D $ be an alternating dimap that is a c-cycle block of size $ m \ge 1 $ and $ S_{t} $ be a tree of size $ t \ge 0 $. Then, \[ T_{c}(D;x,y)=T(S_{m-1};x,y). \] \qed
\end{corollary}

\begin{lemma}\label{lem:c-multiloops}
	Let $ D $ be an alternating dimap of genus zero that is a c-multiloop of size $ m $ and put $ k=\cface(D) $. Then, \[ T_{c}(D;x,y)=y^{m-k}. \]
\end{lemma}

\begin{proof}
	Induction on $ m $.
\end{proof}

\begin{corollary}\label{cor:c-multiloops}
	Let $ D $ be an alternating dimap of genus zero that is a c-multiloop of size $ m $ and put $ k=\cface(D) $, and $ L_{t} $ be a graph with $ t $ loops. Then, \[ T_{c}(D;x,y)=T(L_{m-k};x,y). \] \qed
\end{corollary}

The following two lemmas show the general form of the c-Tutte invariant, when a c-cycle block or a c-multiloop is first reduced in certain alternating dimaps.
\begin{lemma}\label{lem:T_c-c-cycle_blocks}
	Let $ C_{m} $ be a c-cycle block of size $ m \ge 1 $ in a c-alternating dimap $ D $. Then, \[ T_{c}(D;x,y)=x^{m-1} \cdot T_{c}(D\setminus C_{m};x,y). \]
\end{lemma}

\begin{proof}
	Suppose $ C_{m} $ is as stated. We proceed by induction on $ m $. For the base case, suppose $ m=1 $. The c-cycle block $ C_{1} $ is a proper $ \omega^{2} $-loop. By reducing the proper $ \omega^{2} $-loop, we have $ T_{c}(D;x,y)= T_{c}(D[\omega]C_{1};x,y)=x^{1-1} \cdot T_{c}(D\setminus C_{1};x,y) $. Hence, the result for $ m=1 $ follows.
	
	For the inductive step, assume that $ m>1 $ and the result holds for every $ k<m $. Now, every edge in $ C_{m} $ is either a proper $ 1 $-loop or a proper $ \omega $-semiloop. Let $ e \in E(C_{m}) $. By reducing $ e $,
	\begin{align*}
		T_{c}(D;x,y) & = x \cdot T_{c}(D[\omega^{2}]e;x,y) \tag{by (\hyperref[def:T_c]{TC2})}\\
					& = x \cdot x^{(m-1)-1} \cdot T_{c}(D[\omega^{2}]e\setminus C_{m-1};x,y) \tag{by the inductive hypothesis}\\
					& = x^{m-1} \cdot T_{c}(D\setminus C_{m};x,y).
	\end{align*}
	
	This completes the proof, by induction.
\end{proof}

\begin{lemma}\label{lem:T_c-c-multiloops}
	Let $ R_{m} $ be a c-multiloop of size $ m $ in an alternating dimap $ D $ of genus zero and put $ k=\cface(R_{m}) $. Then, \[ T_{c}(D;x,y)=y^{m-k} \cdot T_{c}(D\setminus R_{m};x,y). \]
\end{lemma}

\begin{proof}
	Similar in style to the proof of Lemma~\ref{lem:T_c-c-cycle_blocks}.
\end{proof}

Let $ G_{1} $ and $ G_{2} $ be two graphs. For $ i \in \mathbb{N} \cup \{0\} $, an \emph{i-union} of $ G_{1} $ and $ G_{2} $, denoted by $ G_{1} \cup_{i} G_{2} $, is obtained by identifying exactly $ i $ pairs of vertices $ u_{j},v_{j} $, $ 1 \le j \le i $ such that $ \{u_{1},u_{2},\ldots,u_{i}\} \subseteq V(G_{1}) $ and $ \{v_{1},v_{2},\ldots,v_{i}\} \subseteq V(G_{2}) $.

We now show that the c-Tutte invariant is multiplicative over blocks for any c-simple alternating dimap.
\begin{lemma}\label{lem:multiplicative_c-simple_alternating_dimap}
	Let $ D $ be a c-union of two c-simple alternating dimaps $ S_{1} $ and $ S_{2} $. Then, \[ T_{c}(D;x,y)=T_{c}(S_{1};x,y) \cdot T_{c}(S_{2};x,y).\]
\end{lemma}

\begin{proof}
	Suppose $ D $ is a c-union of two c-simple alternating dimaps $ S_{1} $ and $ S_{2} $. By Corollary~\ref{cor:c-union_c-simple_alt_dimap}, the alternating dimap $ D $ is also a c-simple alternating dimap. Thus, every block of $ D $ is either a c-cycle block or is an element of $ \hbox{alt}_c(\mathcal{G}) $.
	
	We proceed by induction on the number $ p $ of c-cycle blocks of $ D $. For the base case, suppose $ p=0 $, so that there exists no c-cycle block in both $ S_{1} $ and $ S_{2} $. Thus, every block of $ S_{1} $, $ S_{2} $ and hence $ D $ is an element of $ \hbox{alt}_c(\mathcal{G}) $. Let $ S_{1} \cong \hbox{alt}_c(G_{1}) $ and $ S_{2} \cong \hbox{alt}_c(G_{2}) $. Then $ D \cong \hbox{alt}_c(G) $ for some plane graph $ G = G_{1} \cup_{i} G_{2} $, where $  i \in \{0,1\} $. By Theorem~\ref{thm:T_G=T_c_alt_c(G)},
	\begin{equation*}
		T_{c}(D;x,y)=T(G;x,y),
	\end{equation*}
	\begin{equation*}
		T_{c}(S_{i};x,y)=T(G_{i};x,y),\ i \in \{1,2\}.
	\end{equation*}
	Since the Tutte polynomial is multiplicative over blocks for any graph $ G $, we have
	\begin{equation*}
		T(G;x,y)=T(G_{1};x,y) \cdot T(G_{2};x,y).
	\end{equation*}
	Hence,
	\begin{equation*}
		T_{c}(D;x,y)=T_{c}(S_{1};x,y) \cdot T_{c}(S_{2};x,y).
	\end{equation*}
	
	For the inductive step, assume that $ p>0  $ and the result holds for any c-union that contains less than $ p $ c-cycle blocks. Without loss of generality, let $ C_{m} $ be a c-cycle block in $ S_{1} $ that contains $ m \ge 1 $ edges. Since $ D $ is a c-union of $ S_{1} $ and $ S_{2} $, it contains $ C_{m} $ as one of its blocks. By first reducing every edge of $ C_{m} $, we have
	\begin{align*}
		T_{c}(D;x,y) & = x^{m-1} \cdot T_{c}(D\setminus C_{m};x,y) \tag{by Lemma~\ref{lem:T_c-c-cycle_blocks} applied to $ D $}\\
					& = x^{m-1} \cdot T_{c}(S_{1}\setminus C_{m};x,y) \cdot T_{c}(S_{2};x,y) \tag{by the inductive hypothesis}\\
					& = x^{m-1} \cdot \frac{T_{c}(S_{1};x,y)}{x^{m-1}} \cdot T_{c}(S_{2};x,y) \tag{by Lemma~\ref{lem:T_c-c-cycle_blocks} applied to $ S_{1} $}\\
					& = T_{c}(S_{1};x,y) \cdot T_{c}(S_{2};x,y).
	\end{align*}
	
	The result follows, by induction.		
\end{proof}

We extend the result in Lemma~\ref{lem:multiplicative_c-simple_alternating_dimap}, from c-simple alternating dimaps to c-alternating dimaps.
\begin{theorem}\label{thm:multiplicative_c-alternating_dimap}
	Let $ D $ be a c-union of two c-alternating dimaps $ S_{1} $ and $ S_{2} $. Then, \[ T_{c}(D;x,y)=T_{c}(S_{1};x,y) \cdot T_{c}(S_{2};x,y).\]
\end{theorem}

\begin{proof}
	Similar in style to the proof of Lemma~\ref{lem:multiplicative_c-simple_alternating_dimap}.
\end{proof}

Since there are a few non-isomorphic alternating dimaps that may be denoted by $ G_{3,5} $, we write $ G_{3,5,1} $ (see Figure~\ref{fig:G_{3,5,1}}) for the alternating dimap $ G_{3,5} $ that is obtained by subdividing one of the edges of $ \hbox{alt}_c(G) $ where the plane graph $ G $ is a cycle of size exactly two.
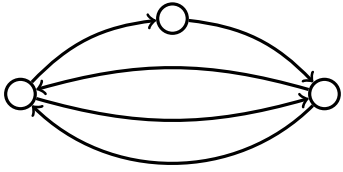
\begin{figure}[t]
	\centering
	\begin{tikzpicture}
	[every path/.style={color=black, line width=1.2pt}, 
	every node/.style={draw, circle, line width=1.2pt},
	every loop/.style={min distance=20mm, in=140, out=220, looseness=20},
	bend angle=45]
	
	\node	(u) at (0,0)	{$ $};
	\node	(w) at (2,1)	{$ $}
		edge[<-, bend right=20]	node[draw=none, above, outer sep=-4pt] {$ $}	(u);
	\node	(v) at (4,0)	{$ $}
		edge[<-, bend right=20]	node[draw=none, above, outer sep=-4pt] {$ $}	(w)
		edge[->, bend right=16]	node[draw=none, below, outer sep=-4pt] {$ $}	(u)
		edge[<-, bend left=16]	(u)
		edge[->, bend left]		(u);
	\end{tikzpicture}
	\caption{Alternating dimap $ G_{3,5,1} $}
	\label{fig:G_{3,5,1}}
\end{figure}

In 2013 (unpublished), Farr proved:
\begin{theorem}\label{thm:minor_farr}
	If $ H \le G $ are alternating dimaps, then there exist $ Y,Z \subseteq E(G) $ such that $ G[\omega]Y[\omega^{2}]Z = H $.
\end{theorem}

\begin{fact}\label{fact:c-cycle block_single_face}
	A block of an alternating dimap is a c-cycle block if and only if the block contains exactly one anticlockwise face and exactly one clockwise face.
\end{fact}

By using Theorem~\ref{thm:minor_farr}, we now show that certain alternating dimaps contain $ G_{3,5,1} $ as a minor.
\begin{lemma}\label{lem:G_{3,5,1}}
	Every non-loop block of an alternating dimap that is neither a c-cycle block nor an element of $ \hbox{alt}_c(\mathcal{G}) $ contains $ G_{3,5,1} $ as a minor.
\end{lemma}

\begin{proof}
	\begin{figure}[H]
		\centering
		\includegraphics[width=0.55\textwidth]{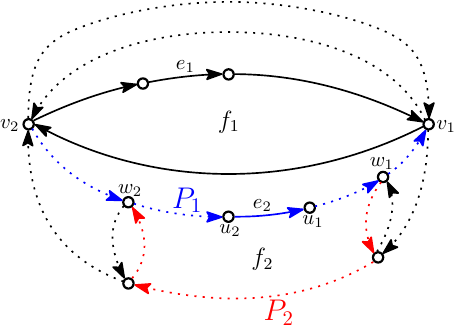}
		\caption{The block $ B $ in the proof of Lemma~\ref{lem:G_{3,5,1}}}
		\label{fig:Minor_G_35}
	\end{figure}

	Let $ B $ be a non-loop block of an alternating dimap that is neither a c-cycle block nor an element of $ \hbox{alt}_c(\mathcal{G}) $. By Fact~\ref{fact:c-cycle block_single_face}, the former implies that the number of a-faces or the number of c-faces of $ B $ is at least two. The existence of at least two a-faces (respectively, c-faces) in a block implies that the number of c-faces (respectively, a-faces) of the block is also at least two. Since $ B $ is a non-loop block, it contains no proper $ \omega^{2} $-loops. By Lemma~\ref{lem:alt_c_size_two}, if $ B $ is an element of $ \hbox{alt}_c(\mathcal{G}) $, every clockwise face of $ B $ has size exactly two. Hence, at least one of the c-faces $ f_{1} $ of $ B $ has size greater than two.
	
	Let $ e_{1} \in E(\partial f_{1}) $. Since $ B $ contains more than one c-face, there exists an edge $ e_{2} $ such that $ e_{2} \notin E(\partial f_{1}) $ and $ e_{2} \in E(\partial f_{2}) $ where $ f_{2} $ is another c-face in $ B $. Since $ B $ contains no cutvertex, there exists a circuit $ C $ that contains both $ e_{1} $ and $ e_{2} $. Let $ u_{1} $ be the head and $ u_{2} $ be the tail of $ e_{2} $. Pick the vertex $ u_{1} $ and traverse $ C $ until the first vertex $ v_{1} \in V(\partial f_{1}) $ is met. Then, pick $ u_{2} $ and traverse $ C $ in the opposite direction and stop once another vertex $ v_{2} \in V(\partial f_{1}) $ is met. Let $ P_{1} $ (highlighted in blue in Figure~\ref{fig:Minor_G_35}) be the path in $ C $ that has $ v_{1} $ and $ v_{2} $ as its endvertices, $ P_{1} $ does not use any edge that belongs to $ f_{1} $ and $ e_{2} \in E(P_{1}) $. Now, observe that vertices $ v_{1} $ and $ v_{2} $ both have degree at least three (the vertices $ v_{1} $ and $ v_{2} $ both belong to $ f_{1} $, and $ P_{1} $ is incident with both of them). Let  $ w_{1},w_{2} \in V(\partial f_{2}) $ and $ P_{2} $ (highlighted in red) be a $ w_{1}w_{2} $-path in $ f_{2} $ such that $ E(P_{2})=E(\partial f_{2})\setminus E(P_{1}) $.
	
	Note that $ f_{1} $ can be contracted to a c-face $ f' $ of size three that contains three vertices (since every c-face of size greater than two can be contracted to a c-face of size exactly two, by Lemma~\ref{lem:bigger_c-face_to_alt_c}). Then, by contracting every edge $ g_{3} \in E(P_{1})\setminus \{e_{2}\} $, the path $ P_{1} $ is reduced to a path of length one that is incident with two of the vertices $ v_{1}', v_{2}' \in V(\partial f') $. Observe that $ P_{2} $ now has $ v_{1}' $ and $ v_{2}' $ as both of its endvertices. Contract every edge in $ P_{2} $ except one, leaving an edge $ e_{3} $ that is incident with $ v_{1}' $ and $ v_{2}' $. Suppose $ E_{1}=E(\partial f') \cup \{e_{2},e_{3}\} $. Then, delete every edge $ h \in E(B)\setminus E_{1} $ to obtain an alternating dimap $ S \cong G_{3,5,1} $. Since $ S \le B $, by Theorem~\ref{thm:minor_farr} there exist $ Y,Z \subseteq E(B) $ such that $ B[\omega]Y[\omega^{2}]Z = S $.
\end{proof}

We next show that certain alternating dimaps contain $ G^{c}_{2,3} $ as a minor. The alternating dimap $ G^{c}_{2,3} $ is shown in Figure~\ref{fig:G_{n,m}}.
 
\begin{lemma}\label{lem:G_{2,3}_c-simple}
	Let $ D $ be an alternating dimap such that there exists a block within a clockwise face of some other block and they form a clockwise face of size greater than two. Then $ D $ has $ G^{c}_{2,3} $ as a minor.
\end{lemma}

\begin{proof}
	\begin{figure}[ht]
		\centering
		\includegraphics[width=0.4\textwidth]{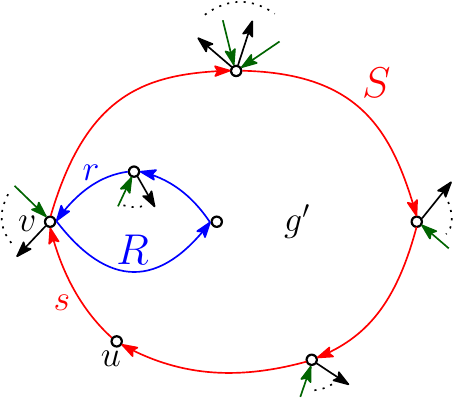}
		\caption{Two blocks $ B_{1} $ and $ B_{2} $ in the proof of Lemma~\ref{lem:G_{2,3}_c-simple}}
		\label{fig:G_2,c3}
	\end{figure}
	
	Suppose an alternating dimap $ D $ contains two blocks $ B_{1} $ and $ B_{2} $ such that these two blocks share exactly one common vertex $ v $, and $ B_{1} $ is within one of the c-faces $ g $ of $ B_{2} $. Let $ g' \in F(D) $ be the c-face of size greater than two that is formed by the boundary of $ g $ and some edges of $ B_{1} $.
	
	Since $ g' $ has size greater than two, it contains at least two vertices including $ v $. The fact that $ v $ has degree greater than two in $ D $ implies that every edge $ e_{1} \in I(v) $ is not a proper $ 1 $-loop. Suppose $ R=E(B_{1}) \cap E(\partial g') $ and $ S=E(B_{2}) \cap E(\partial g') $. Note that $ S $ is the boundary of the face $ g $ in $ B_{2} $. Let $ r \in R \subseteq E(B_{1}) $ and $ s \in S \subseteq E(B_{2}) $ and $ r,s \in E(\partial g') \cap I(v) $. Since $ g' $ is formed by at least three edges, at least one of the two edges $ r $ and $ s $ is a non-loop edge (otherwise $ g' $ is a clockwise face of size two). Without loss of generality, let $ s=uv $ be a non-loop edge. Since $ s \in I(v) $ and is a non-loop edge, $ s $ is a non-triloop edge in $ D $. By Corollary~\ref{cor:minors}, the alternating dimap $ D $ has $ G_{1,3} $ or $ G_{2,3} $ as its minor.
	
	To obtain $ G^{c}_{2,3} $ as a minor in $ D $, for each vertex $ w \in V(\partial g') $ and for each edge $ e_{2} \in I(w) \setminus E(\partial g') $ (green edges in Figure~\ref{fig:G_2,c3}), reduce $ e_{2} $ using $ \omega^{2} $-reduction. Now, $ g' $ belongs to a component $ P $ that has exactly two blocks. Let $ T=I(u) \cup I(v) $ in $ P $. By contracting every edge $ e_{3} \in E(P) \setminus T $, we obtain a component $ P' \cong G^{c}_{2,3} $. Since $ P' \le D $, by Theorem~\ref{thm:minor_farr}, there exist $ Y,Z \subseteq E(D) $ such that $ D[\omega]Y[\omega^{2}]Z = P' $.
	
	Therefore, $ D $ has $ G^{c}_{2,3} $ as a minor.
\end{proof}

\begin{lemma}\label{lem:not_well_defined_G{2,c3}}
	The c-Tutte invariant is not well defined for the alternating dimap $ G^{c}_{2,3} $.
\end{lemma}

\begin{proof}[Proof (sketch)]
	By reducing the alternating dimap $ G^{c}_{2,3} $ using different edge-orderings, we obtain two distinct derived polynomials $ xy $ and $ 1 $.
\end{proof}

By using a similar approach, we show that the c-Tutte invariant is not well defined for alternating dimap $ G_{3,5,1} $.

\begin{lemma}\label{lem:not_well_defined_G{3,5}}
	The c-Tutte invariant is not well defined for the alternating dimap $ G_{3,5,1} $.
\end{lemma}

\begin{proof}[Proof (sketch)]
	By reducing the alternating dimap $  G_{3,5,1} $ using different edge-orderings, we obtain two distinct derived polynomials $ x^{2} + xy $ and $ x^{2}+x+y $.
\end{proof}

\begin{theorem}\label{thm:well defined_c-Tutte}
	The c-Tutte invariant is well defined for an alternating dimap $ D $ if and only if $ D $ is a c-alternating dimap.
\end{theorem}

\begin{proof}
	The forward implication is proved by contradiction using two different cases.
	
	Let $ D $ be an alternating dimap such that the c-Tutte invariant is well defined for $ D $. Suppose $ D $ is not a c-alternating dimap. This implies that $ D $ is not a c-simple alternating dimap after every loop in $ D $ is removed. Thus, either it contains a block that is neither a c-cycle block nor an element of $ \hbox{alt}_c(\mathcal{G}) $, or there exists a block within a clockwise face of some other block.
	
	First, assume that $ D $ contains a block $ B $ that is neither a c-cycle block nor an element of $ \hbox{alt}_c(\mathcal{G}) $. By Lemma~\ref{lem:G_{3,5,1}}, the block $ B $ contains $ G_{3,5,1} $ as a minor. By Lemma~\ref{lem:not_well_defined_G{3,5}}, the c-Tutte invariant is not well defined for $ G_{3,5,1} $. Since $ B $ is a block of $ D $, the alternating dimap $ D $ contains $ G_{3,5,1} $ as a minor. By Proposition~\ref{pro:well_defined}, the c-Tutte invariant is not well defined for $ D $. We reach a contradiction.
	
	Secondly, suppose there exists a block $ B $ in $ D $ such that $ B $ contains another block $ B' $ within one of its clockwise faces. Note that $ B $ and $ B' $ form a clockwise face of size greater than two, else it is a c-multiloop. By Lemma~\ref{lem:G_{2,3}_c-simple}, the alternating dimap $ D $ contains $ G^{c}_{2,3} $ as a minor. By Lemma~\ref{lem:not_well_defined_G{2,c3}}, the c-Tutte invariant is not well defined for $ G^{c}_{2,3} $. By Proposition~\ref{pro:well_defined}, we again get a contradiction. Hence, the forward implication follows.
	
	It remains to show if $ D $ is a c-alternating dimap, then the c-Tutte invariant is well defined for $ D $. Every non-loop block of $ D $ is either a c-cycle block or is an element of $ \hbox{alt}_{c}(\mathcal{G}) $, and $ D $ contains no block within a clockwise face of any other block. In addition, $ D $ may contain some c-multiloops within some of its anticlockwise faces. By Lemma~\ref{lem:c-cycle_blocks}, the c-Tutte invariant is well defined for every c-cycle block. By Theorem~\ref{thm:T_G=T_c_alt_c(G)}, the c-Tutte invariant is also well defined for alternating dimaps that belongs to $ \hbox{alt}_{c}(\mathcal{G}) $. By Lemma~\ref{lem:c-multiloops}, the c-Tutte invariant is also well defined for every c-multiloop. By Theorem~\ref{thm:multiplicative_c-alternating_dimap}, the c-Tutte invariant is multiplicative over non-loop blocks and c-multiloops for any c-alternating dimap. Hence, the c-Tutte invariant is well defined for $ D $.
	
	Therefore, the c-Tutte invariant is well defined for an alternating dimap $ D $ if and only if $ D $ is a c-alternating dimap.		
\end{proof}

We now develop a relationship between plane graphs $ G $ and c-alternating dimaps $ D $, when the Tutte polynomial of $ G $ and the c-Tutte invariant of $ D $ are both identical.

\begin{theorem}\label{thm:Tutte_polynomial_vs_c-Tutte}
	Let $ D $ be a c-alternating dimap, $ G $ be a plane graph and $ G' $ is obtained from $ G $ by deleting all the loops and bridges in $ G $. Let $ R $ and $ S $ be the set of c-cycle blocks and c-multiloops in $ D $, respectively. Let $ E_{R} = \bigcup_{r \in R} E(r) $, $ E_{S} = \bigcup_{s \in S} E(s) $ and $ D'=D\setminus E_{R}\setminus E_{S} $. Then, $ T(G;x,y) = T_{c}(D;x,y) $ if and only if $ T_{c}(\hbox{alt}_c(G');x,y) = T_{c}(D';x,y) $ and $ G $ contains $ \sum_{r \in R} \left( \abs{r}-1 \right) $ bridges and $ \sum_{s \in S} \left( \abs{s}-\cface(s) \right) $ loops.
\end{theorem}

\noindent
\textbf{Remark}: $ T(G;x,y) = T_{c}(D;x,y) $ if and only if $ T_{c}(\hbox{alt}_c(G);x,y) = T_{c}(D;x,y) $. So Theorem~\ref{thm:Tutte_polynomial_vs_c-Tutte} is about situations of $ T_{c} $-equivalence in alternating dimaps.

\begin{proof}
	Let $ D $, $ G $, $ G' $, $ R $, $ S $, $ E_{R} $ and $ E_{S} $ be as stated. To prove the forward implication, we let $ B $ and $ L $ be the sets of bridges and loops in $ G $, respectively. Suppose $ \abs{B}=p $ and $ \abs{L}=q $. Since the Tutte polynomial is multiplicative over blocks, we have
	\begin{equation*}
		T_{c}(D;x,y) = T(G;x,y) = x^{p} \cdot y^{q} \cdot T(G';x,y).
	\end{equation*}
	By Theorem~\ref{thm:multiplicative_c-alternating_dimap}, the c-Tutte invariant is multiplicative over c-cycle blocks, elements of $ \hbox{alt}_c(\mathcal{G}) $ and c-multiloops. By Definition~\ref{def:T_c}, a factor of $ x $ is introduced when a proper $ 1 $-loop or a proper $ \omega $-semiloop is reduced. In any c-alternating dimap, proper $ 1 $-loops and proper $ \omega $-semiloops can only be found in c-cycle blocks. Note that if a plane graph $ H $ has a single edge, then $ \hbox{alt}_c(H) $ is also a c-cycle block of size two. By Lemma~\ref{lem:T_c-c-cycle_blocks} and using the fact that the c-Tutte invariant is multiplicative over c-cycle blocks, we have $ p = \sum_{r \in R} \left( \abs{r}-1 \right) $. Likewise, a factor of $ y $ is introduced when a proper $ \omega $-loop or a proper $ 1 $-semiloop is reduced. These two types of edges can only be found in c-multiloops. By Lemma~\ref{lem:T_c-c-multiloops} and using the fact that the c-Tutte invariant is multiplicative over c-multiloops, we have $ q = \sum_{s \in S} \left( \abs{s}-\cface(s) \right) $. After each c-cycle block and each c-multiloop is reduced in $ D $, we obtain $ D'=D\setminus E_{R}\setminus E_{S} $. It can also be seen that $ D' $ is an element of $ \hbox{alt}_c(\mathcal{G}) $. Since
	\begin{equation*}
		x^{p} \cdot y^{q} \cdot T(G';x,y) = T_{c}(D;x,y) = x^{p} \cdot y^{q} \cdot T_{c}(D';x,y),
	\end{equation*}
	we have
	\begin{equation}\label{eqn:Tutte_polynomial_vs_c-Tutte}
		T(G';x,y) = T_{c}(D';x,y).
	\end{equation}
	Since $ G' $ is a plane graph, by Theorem~\ref{thm:T_G=T_c_alt_c(G)} and (\ref{eqn:Tutte_polynomial_vs_c-Tutte}), we obtain
	\begin{equation*}
		T_{c}(\hbox{alt}_c(G');x,y) = T(G';x,y) = T_{c}(D';x,y).
	\end{equation*}	
	
	The backward implication follows immediately, by Theorem~\ref{thm:T_G=T_c_alt_c(G)}, Lemmas~\ref{lem:T_c-c-cycle_blocks} and \ref{lem:T_c-c-multiloops}, and Theorem~\ref{thm:multiplicative_c-alternating_dimap}.
\end{proof}

\begin{corollary}
	If $ D $ is a c-alternating dimap, then there exists a plane graph $ H $ such that $ T_{c}(D;x,y) = T_{c}(\emph{\hbox{alt}}_c(H);x,y) = T(H;x,y) $.
\end{corollary}

\begin{proof}
	Let $ D $ be a c-alternating dimap, $ R $ and $ S $ be the set of c-cycle blocks and c-multiloops in $ D $, respectively. Let $ E_{R} = \bigcup_{r \in R} E(r) $, $ E_{S} = \bigcup_{s \in S} E(s) $ and $ D'=D\setminus E_{R}\setminus E_{S} $. Every block in $ D' $ is an element of $ \hbox{alt}_c(\mathcal{G}) $. This implies that every clockwise face of $ D' $ has size exactly two. By Lemma~\ref{lem:alt_c_size_two}, there exists an undirected graph $ H' $ such that $ D' \cong \hbox{alt}_{c}(H') $. By adding in $ \sum_{r \in R} \left( \abs{r}-1 \right) $ bridges and $ \sum_{s \in S} \left( \abs{s}-\cface(s) \right) $ loops into $ H' $, we obtain a plane graph $ H $ and $ T_{c}(D;x,y) = T_{c}(\hbox{alt}_c(H);x,y) = T(H;x,y) $.
\end{proof}

By defining an \emph{a-alternating dimap} with appropriate modifications, we have the following two corollaries, based on Theorem~\ref{thm:well defined_c-Tutte} and Theorem~\ref{thm:Tutte_polynomial_vs_c-Tutte}, respectively.

\begin{corollary}
	The a-Tutte invariant is well defined for an alternating dimap $ D $ if and only if $ D $ is an a-alternating dimap. \qed
\end{corollary}

\begin{corollary}
	Let $ D $ be an a-alternating dimap, $ G $ be a plane graph and $ G' $ is obtained from $ G $ by deleting all the loops and bridges in $ G $. Let $ R $ and $ S $ be the set of a-cycle blocks and a-multiloops in $ D $, respectively. Let $ E_{R} = \bigcup_{r \in R} E(r) $, $ E_{S} = \bigcup_{s \in S} E(s) $ and $ D'=D\setminus E_{R}\setminus E_{S} $. Then, $ T(G;x,y) = T_{a}(D;x,y) $ if and only if $ T_{a}(\hbox{alt}_a(G');x,y) = T_{a}(D';x,y) $ and $ G $ contains $ \sum_{r \in R} \left( \abs{r}-1 \right) $ bridges and $ \sum_{s \in S} \left( \abs{s}-\aface(s) \right) $ loops. \qed
\end{corollary}

\section{Concluding Remarks}

We conclude this article with some problems for further research.

Recall that in \cref{sec:characterisations_of_ETI}, we only discuss extended Tutte invariants for alternating dimaps that are embedded on an orientable surface that has genus zero. The reason is, we believe that more edge types may need to be defined if an alternating dimap is embedded on an orientable surface that has genus greater than zero. For instance, consider a $ k $-posy for $ k \ge 1 $ (see \cite{Farr2018} for more details), and a minor of $ \hbox{alt}_c(K_4) $ that is embedded on a torus (see Figure~\ref{fig:Torus}). (The minor of $ \hbox{alt}_c(K_4) $ that is shown in Figure~\ref{fig:Torus} can be obtained by performing four $ 1 $-reductions and four $ \omega $-reductions on the edges of $ \hbox{alt}_c(K_4) $.) These alternating dimaps each have edges that are proper $ \mu $-semiloops for two or three distinct $ \mu $, a situation that cannot occur in the plane and which increases the ways in which a Tutte invariant may fail to be well defined.
\begin{figure}
	\centering
	\begin{subfigure}[t]{0.20\textwidth}
		\includegraphics[width=\textwidth]{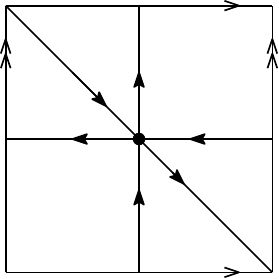}
		\caption{}
	\end{subfigure}\qquad
	~
	\begin{subfigure}[t]{0.20\textwidth}
		\includegraphics[width=\textwidth]{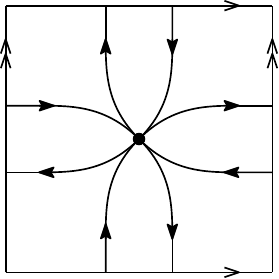}
		\caption{}
	\end{subfigure}
	\caption{(a) A 1-posy, (b) A minor of $ \hbox{alt}_c(K_4) $}
	\label{fig:Torus}
\end{figure}

\begin{enumerate}
	\item How do we properly define extended Tutte invariants for alternating dimaps that are embedded on an orientable surface with genus greater than zero? 
\end{enumerate}

Knowing that the c-Tutte invariant is an analogue of the Tutte polynomial under certain circumstances, a natural question that arises is:
\begin{enumerate}[resume]	
	\item Does the c-Tutte invariant yield another option to compute the Tutte polynomial for abstract planar graphs?
	\item Is the c-Tutte invariant well defined for some alternating dimaps that are embedded on an orientable surface with genus greater than zero? If so, can we characterise them?
	\item \label{Iain} Can we evaluate the c-Tutte invariant in terms of the ribbon graph polynomial of Bollob\'{a}s and Riordan (cf. Theorem~\ref{thm:T_G=T_c_alt_c(G)})? (Suggested by Iain Moffatt.)
\end{enumerate}

\end{document}